\documentclass[a4paper,11pt]{article}
\usepackage{amssymb,amsmath,amsthm,amsfonts}
\usepackage{bookmark}
\usepackage{mathrsfs}
\usepackage{multirow}
\usepackage{graphicx}
\usepackage{authblk}
\usepackage{indentfirst}
\usepackage{multicol}
\usepackage{tabu}
\usepackage{url}
\usepackage{fancyhdr}
\usepackage[numbers]{natbib}
\usepackage[all]{xy}
\usepackage{mathtools}
\usepackage[english]{babel}
\usepackage{colortbl}
\usepackage{caption}
\usepackage{hyperref}\usepackage{subcaption}
\usepackage{float}\usepackage{enumitem}\usepackage{xcolor}\usepackage{tikz-cd}\usepackage{bm}
\usepackage{tikz}
\usetikzlibrary{graphs, positioning, arrows.meta, shapes.geometric, calc}

\newtheorem{theorem}{Theorem}[section]

\newtheorem{lemma}[theorem]{Lemma}
\newtheorem{proposition}[theorem]{Proposition}

\theoremstyle{definition}
\newtheorem{definition}[theorem]{Definition}
\newtheorem{example}[theorem]{Example}
\newtheorem{remark}[theorem]{Remark}

\DeclareMathOperator{\im}{im}

\DeclareMathOperator{\clst}{ClSt}
\DeclareMathOperator{\Lk}{Lk}
\DeclareMathOperator{\st}{St}

\DeclareMathOperator{\Clq}{Clq}

\definecolor{myblue}{RGB}{50, 110, 190}
\definecolor{mygreen}{RGB}{40, 160, 80}
\definecolor{myred}{RGB}{210, 60, 60}
\definecolor{mypurple}{RGB}{140, 60, 180}
\title{Local Laplacian: theory and models for data analysis}

\author[1]{Jian Liu}%
\author[2]{Hongsong Feng}
\author[1]{Kefeng Liu\thanks{Corresponding author: kefengliu@cqut.edu.cn}}
\affil[1]{Mathematical Science Research Center, Chongqing University of Technology, Chongqing 400054, China}
\affil[2]{Department of Mathematics and Statistics, University of North Carolina at Charlotte, Charlotte, NC 28223, USA}

\makeatletter
    \renewcommand*{\@fnsymbol}[1]{\ensuremath{\ifcase#1\or \dagger\or *\or *\or
   \mathsection\or \else\@ctrerr\fi}}
\makeatother
\date{}

\begin{document}
    \maketitle
    
    \paragraph{Abstract}
    
    While topological data analysis has emerged as a powerful paradigm for structural inference, its foundational tools, notably persistent homology and the persistent Laplacian, are frequently insensitive to localized structural fluctuations and suffer from prohibitive computational costs on large-scale datasets. To bridge this gap, we introduce the persistent local Laplacian formalism, which is designed to extract fine-grained local topological and geometric signatures while enabling a highly efficient, parallelizable computational workflow. On the theoretical front, we prove a generalized persistent Hodge isomorphism, establishing that the harmonic space of the persistent local Laplacian is isomorphic to the persistent local homology. Furthermore, we derive a unitary equivalence between the persistent local Laplacian and the persistent Laplacian of its corresponding link complex at a shifted dimension. This spectral conjugacy establishes the mathematical foundation for developing efficient computational schemes to resolve persistent local spectral invariants. We further extend this construction to point clouds and graph-structured data, characterizing their persistent local spectral properties through combinatorial filtrations. The resulting architecture is inherently decoupled, facilitating massive parallelization and rendering it uniquely scalable for large-scale network analysis and distributed computational environments.

    \paragraph{Keywords}
     Persistent homology, local homology, persistent local Laplacian, Hodge isomorphism, link complex.

\footnotetext[1]
{ {\bf 2020 Mathematics Subject Classification.}  	Primary 55N31; Secondary 58A14, 15A63.
}

\tableofcontents 

\section{Introduction}

On a manifold, singular, simplicial, and de Rham homology theories are naturally isomorphic, each offering a unique yet equivalent characterization of the underlying topology. Collectively, these theories interpret the same topological invariants through topological, combinatorial, and geometric prisms. While topological homology remains metric-independent and combinatorial approaches often lack local resolution, geometry provides a more nuanced description of the space. This geometric depth is most elegantly captured by the Laplacian operator: the Hodge Theorem identifies the harmonic space of the Laplacian with the manifold's homology, whereas the non-harmonic spectrum encodes the finer details of the manifold's asymmetric geometry and underlying metric structure.

In modern data science, simplicial complexes have emerged as a predominant topological representation for modeling data and objects. Although they are not manifolds in the strictest sense, they function as piecewise linear manifolds or combinatorial analogs. While the extraction of topological invariants from these complexes is well-established, the characterization of their intrinsic geometric features remains an ongoing challenge. Simplicial complexes admit a combinatorial Laplacian \cite{eckmann1944harmonische, hirani2003discrete, horak2013spectra}, which generalizes the classical graph Laplacian to higher dimensions \cite{lim2020hodge}. However, the standard combinatorial Laplacian often fails to resolve fine-grained local information. 

In contrast, local homology provides a refined, localized lens for characterizing the topological features of a space around a specific point \cite{hatcher2002algebraic, munkres2018elements}. The theoretical foundation of local homology is grounded in the excision theorem, which establishes a fundamental principle: the topological profile of a point is locally constrained. Specifically, this profile is determined solely by the point's infinitesimal surroundings, rendering it invariant to the global configuration of the ambient space. This progression is driven by two factors: analytical rigor and the practical demands of data-driven applications. A natural confluence of these ideas is the introduction of a local Laplacian on simplicial complexes. This construction serves a clear purpose: it bridges the gap between local topology and localized geometry.

In data analysis, static invariants often lack sufficient descriptive power to resolve complex structural nuances. A pivotal advancement in overcoming this limitation is the paradigm of \textit{``persistence''}, which facilitates the tracking of topological and geometric features across multiple scales. Persistent homology, introduced by Carlsson, Edelsbrunner, and others, characterizes multiscale data structures by tracking the evolution of connected components, loops, and voids \cite{cohen2007stability, edelsbrunner2008persistent, zomorodian2004computing}. Recently, this multiscale philosophy has been extended to spectral theory, leading to the development of the persistent Laplacian and its associated spectral invariants \cite{chen2021evolutionary, memoli2022persistent, wang2020persistent}.

Despite the efficacy of global persistent homology and Laplacians in capturing macroscopic data topology, these approaches often remain insensitive to the local structural nuances that define the fine-grained geometry of data. To bridge this gap, persistent local homology was introduced \cite{fasy2016exploring}, with subsequent research extending its utility to graph-structured data and neural network architectures \cite{cesa2023algebraic, wang2024persistent}. Nevertheless, a systematic study of the persistent local Laplacian remains conspicuously absent from the literature. We propose that the persistent local Laplacian is essential for addressing the complexities of hierarchical data processing, providing a high-resolution, multiscale signature of local topological and geometric evolutions.

\begin{figure}[htb!]
    \centering
    \begin{tikzpicture}[
        >=Stealth, thick,
        box/.style={
            rectangle, draw, rounded corners=3pt, 
            minimum width=5.0cm, minimum height=1.4cm, 
            align=center, font=\small\sffamily, thick
        },
        arrow label/.style={
            font=\scriptsize\bfseries, 
            align=center, 
            midway
        }
    ]

        \definecolor{myblue}{RGB}{50, 110, 190}
        \definecolor{mygreen}{RGB}{40, 160, 80}
        \definecolor{myred}{RGB}{210, 60, 60}
        \definecolor{mypurple}{RGB}{140, 60, 180}

        \node[box, draw=myblue!80, fill=myblue!5] (static0d) at (0,0) {
            \textbf{Classical Graph Laplacian} \\
            $\Delta_0 = d_1 d_1^\ast = L$
        };
        
        \node[box, draw=mygreen!80, fill=mygreen!5] (static1d) at (0, 4.5) {
            \textbf{Combinatorial Laplacian} \\
            $\Delta_n = d_n^\ast d_n + d_{n+1} d_{n+1}^\ast$
        };
        
        \node[box, draw=myred!80, fill=myred!5] (dynamic0d) at (7.5, 0) {
            \textbf{Persistent Laplacian} \\
            $\Delta_0^{i,j} = \delta_1^{i,j} (\delta_1^{i,j})^\ast$
        };
        
        \node[box, draw=mypurple!90, fill=mypurple!8, thick] (pll) at (7.5, 4.5) {
            \textbf{Persistent Local Laplacian} \\
             $\text{Localization} + \text{Persistence} + \text{Higher-D}$
        };

        \draw[->, mygreen, line width=1.5pt] (static0d.north) -- (static1d.south) 
            node[arrow label, left=5pt, text=black] {Higher-D \\ ($n$-complex)};

        \draw[->, myred, line width=1.5pt] (static0d.east) -- (dynamic0d.west)
            node[arrow label, below=5pt, text=black] {Persistence \\ (Filtration)};


        \draw[->, mypurple, line width=2pt, dotted] (static0d.north east) .. controls (3.75, 1.0) and (3.75, 3.5) .. (pll.south west);

       \draw[->, mypurple, line width=1.5pt, dashed] (static1d.east) -- (pll.west)
            node[arrow label, above=3pt, text=black] {Persistence}
            node[arrow label, below=3pt, text=black] {Localization};

        \draw[->, mypurple, line width=1.5pt, dashed] (dynamic0d.north) -- (pll.south)
            node[arrow label, right=3pt, text=black] {Higher-D}
            node[arrow label, left=3pt, text=black] {Localization};
    \end{tikzpicture}
    \caption{Schematic of the Laplacian family: from classical graph Laplacian to persistent local Laplacian.}
\end{figure}
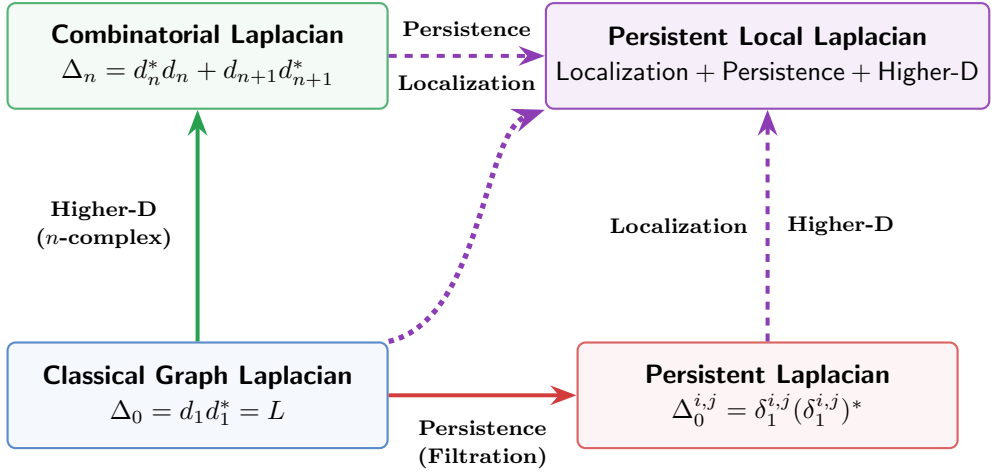

From an application perspective, the significance of the persistent local Laplacian extends well beyond mere conceptual generalization. The inherent computational bottlenecks of standard persistent homology and global persistent Laplacians frequently constrain their applicability to large-scale datasets. The persistent local Laplacian addresses these limitations by providing two primary operational advantages:
\begin{itemize}
    \item \textbf{Computational Efficiency:} Restricting spectral analysis to local subcomplexes (links) significantly reduces memory overhead and algorithmic complexity.
    \item \textbf{Inherent Scalability:} Local feature extraction is naturally decoupled, facilitating massive parallelization for large-scale network analysis and distributed computing.
\end{itemize}

Let $K$ be a simplicial complex with its associated chain complex $C_{\ast}(K)$. The classical $n$-th combinatorial Laplacian $\Delta_n\colon C_{n}(K)\to C_{n}(K)$ is defined as
\begin{equation*}
  \Delta_n = d_{n+1} d_{n+1}^\ast + d_n^\ast d_n,
\end{equation*}
where $d_{\ast}$ denotes the boundary operator. For $n=0$, $\Delta_0$ coincides with the graph Laplacian of the 1-skeleton of $K$. Notably, for a graph $G$, the Laplacian of its clique complex $\mathrm{Clq}(G)$ at dimension zero recovers the standard graph Laplacian \cite{kozlov2008combinatorial}. This property ensures that higher-dimensional spectral information is intrinsically preserved in graph-based data.

Consider a simplicial embedding $f: K \hookrightarrow L$, which induces an injective chain map $f_\ast: C_{\ast}(K) \hookrightarrow C_{\ast}(L)$. This gives rise to the persistent Laplacian
\begin{equation*}
\Delta_n^f = \delta_{f_n} \delta_{f_n}^* + d_n^* d_n,
\end{equation*}
where the term $\delta_{f_n} \delta_{f_n}^*$ characterizes the ``persistence''. For a general simplicial map $f: K \to L$ that is not necessarily an embedding, we introduce the generalized persistent Laplacian
\begin{equation*}
\Delta_n^f = \delta_{f_n} \delta_{f_n}^* + d_n^* d_n + (\mathrm{id} - f_n^\dagger f_n),
\end{equation*}
where $f_n^\dagger$ denotes the Moore-Penrose pseudoinverse of $f_n$. Consequently, we establish a generalized persistent Hodge isomorphism: the kernel of the persistent Laplacian $\ker \Delta_n^f$ is isomorphic to the persistent homology $H_n^f$ (see Theorem \ref{theorem:hodge_isomorphism}).

Transitioning to the localized setting, let $v$ be a vertex in $K$. The local homology at $v$ is characterized by its \emph{link complex}, $\mathrm{Lk}_K(v)=\{\sigma \in K \mid v\notin \sigma,\sigma \cup \{v\} \in K\}$, via the fundamental isomorphism
\begin{equation*}
    H_n(K, K \setminus \{v\}) \cong \tilde{H}_{n-1}(\mathrm{Lk}_K(v)).
\end{equation*}
Crucially, the link complex also characterizes the \emph{local Laplacian}. Specifically, there exists a \emph{unitary equivalence} between the local Laplacian and the $(n-1)$-th Laplacian of the link complex (Theorem \ref{theorem:laplacian_conjugacy}):
\begin{equation*}
    \Delta_n^{K,v} \cong \Delta_{n-1}^{\mathrm{Lk}}.
\end{equation*}
Extending this to the persistence setting, we establish a unitary equivalence (Theorem \ref{theorem:persistence_link}):
\begin{equation*}
    \Delta_n^{i,j} \cong \Delta_{n-1}^{\mathrm{Lk},i,j},
\end{equation*}
which identifies the persistent local Laplacian with the persistent Laplacian of the corresponding link complex at a shifted dimension. This result shows that the computation of persistent local Laplacians can be reduced to the construction and spectral analysis of the corresponding persistent link complexes.

For practical data analysis, we consider point clouds and graph-structured data, where the Vietoris-Rips and clique complexes exhibit canonical filtrations. We provide combinatorial characterizations of their link complexes, facilitating the efficient computation of the persistent local Laplacian.

Overall, this work provides a systematic investigation into the persistent local Laplacian, establishing the mathematical foundation for its algorithmic computation. We provide formal definitions and structural results at the algebraic level, complemented by the explicit construction and characterization of the persistent local Laplacian for data-driven applications. We hope this paper will be of interest to researchers in topological data analysis, combinatorial geometry, and spectral graph theory.

The remainder of this paper is organized as follows. In the next section, we review the fundamental concepts and computational aspects of local homology and local Laplacians. Section \ref{section:generalized_Laplacian} establishes the theoretical framework for the persistent local Laplacian, providing its formal definition and associated results. Finally, in Section \ref{section:on_datasets}, we detail the construction and computation of the persistent local Laplacian on various datasets, including point clouds and graph-structured data.

\section{Local Laplacians on simplicial complexes}\label{section:preliminaries}

Local homology is a classical tool for characterizing the localized topological features of a space. By capturing relative homological data at a specific site, it provides a rigorous measure of local manifold structures and singularities. Building on this topological intuition, we introduce the local Laplacian, an operator that encapsulates both harmonic and non-harmonic combinatorial information of a simplicial complex within a localized regime, thereby extending traditional spectral analysis to a local perspective.

Throughout this work, we fix a base field $\mathbb{K}$, which will typically be $\mathbb{R}$ or $\mathbb{C}$.

\subsection{Local homology of simplicial complexes}

\subsubsection{Relative homology}

Let $K$ be a simplicial complex and $L \subseteq K$ a subcomplex. The $n$-th \textit{relative chain group}, denoted by $C_n(K, L)$, is defined as the quotient of the absolute chain groups
\[
C_n(K, L) := C_n(K) / C_n(L),
\]
where $C_n(K)$ and $C_n(L)$ are the free modules generated by the $n$-simplices of $K$ and $L$ respectively, over a field $\mathbb{K}$.

Since the boundary operator $d_n\colon C_n(K) \to C_{n-1}(K)$ maps $C_n(L)$ into $C_{n-1}(L)$, it naturally induces a quotient homomorphism on relative chains
\[
\bar{d}_n\colon C_n(K, L) \to C_{n-1}(K, L).
\]
It follows immediately from the relation $d^2 = 0$ that $\bar{d}^2 = 0$, thus yielding a chain complex $(C_*(K, L), \bar{d}_*)$.
\begin{definition}
The $n$-th \textit{relative homology group} is defined as the homology of this quotient complex
\[
H_n(K, L) := \frac{\ker \bar{d}_n}{\im \bar{d}_{n+1}}.
\]
\end{definition}

Intuitively, $H_n(K, L)$ characterizes the topological cycles of $K$ whose boundaries are contained within $L$, effectively collapsing the subcomplex $L$ to a point and highlighting the homological features of the pair $(K, L)$ that are extrinsic to $L$.

\begin{remark}
For a topological pair $(X, Y)$ with $Y \subseteq X$, one can define the relative singular homology $H_n(X, Y)$. Correspondingly, any simplicial complex $K$ possesses a geometric realization $|K|$, which is a compact Hausdorff space. For any subcomplex $L \subseteq K$, there is a naturally isomorphism
\begin{equation*}
    H_n(K, L) \xrightarrow{\cong} H_n(|K|, |L|), \quad \forall n \ge 0
\end{equation*}
between the relative simplicial homology and the relative singular homology.
\end{remark}

\subsubsection{Local homology}

Let $K$ be a simplicial complex, and let $v \in K$ be a vertex.
\begin{definition}
The \textit{local homology} of $K$ at $v$ at dimension $n$, denoted by $H_n(|K|, |K| \setminus \{v\})$, is defined as the $n$-th homology of the relative chain complex derived from the inclusion of the deleted subcomplex $|K| \setminus \{v\} \hookrightarrow |K|$.
\end{definition}

From a structural perspective, $H_n(|K|, |K| \setminus \{v\})$ measures the obstruction to $v$ being ``homologically trivial'' within its immediate neighborhood. Due to the excision property of relative homology, these groups are purely local invariants: for any open neighborhood $U$ of $v$ in the geometric realization $|K|$, the excision theorem yields a canonical isomorphism
\[
H_n(|K|, |K| \setminus \{v\}) \cong H_n(U, U \setminus \{v\}).
\]

The local homology groups $H_n(|K|, |K| \setminus \{v\})$ provide a formal characterization of the local topological type at a vertex $v$, effectively decoupling the site-specific invariants from the global structure of the space. By leveraging the excision property to isolate the contribution of the vertex, these groups serve as algebraic probes for detecting local singularities and identifying manifold-like configurations within the combinatorial data.

\subsubsection{Closed star, link, and open star}

The \textit{closed star} of $x$, denoted $\clst_K(v)$, is the subcomplex of all simplices in $K$ containing $v$ and their faces
\[
\clst_K(v) = \{ \tau \in K \mid \exists \sigma \in K \text{ with } v \in \sigma \text{ and } \tau \subseteq \sigma \}.
\]
The \textit{link} of a point $v$ in $K$ is the set of simplices in $\clst_K(v)$ that do not contain $v$, given by
\[
\Lk_K(v) = \{ \tau \in \clst_K(v) \mid v \notin \tau \}.
\]
The \textit{open star} $\st_K(v) \subset |K|$ is the union of the interiors of simplices containing $v$, an open set in $|K|$. For a simplex $\sigma \in K$, its interior in $|K|$ (under the weak topology) is the set of points in $|K|$ where the barycentric coordinate corresponding to $\sigma$ is positive. The open star is
\[
\st_K(v) = \bigcup_{\sigma \in K, v \in \sigma} \text{int}(\sigma),
\]
where $\text{int}(\sigma)$ is the interior of $\sigma$ in $|K|$. In the geometric realization, $\st_K(v)$ is an open neighborhood of $v$, as it excludes the boundaries of simplices containing $v$.

By the excision theorem, for the open star $\st_K(v)$, we have
\[
H_n(|K|, |K| \setminus \st_K(v)) \cong H_n(\st_K(v), \st_K(v) \setminus \{v\}).
\]
Furthermore, we obtain
\[
H_n(|K|, |K| \setminus \st_K(v)) \cong H_n(\clst_K(v), \Lk_K(v)).
\]
Note that $\Lk_K(v)$ is contractible. By the long exact sequence of homology, we have the isomorphism
\[
H_n(|K|, |K| \setminus \st_K(v)) \cong \tilde{H}_{n-1}(\Lk_K(v)),
\]
where $\tilde{H}_*$ denotes reduced homology. This isomorphism provides a direct method for calculating local homology.

\subsubsection{Computing local homology}

Let $K$ be a simplicial complex, and let $v$ be a vertex of $K$. The local homology groups $H_n(|K|, |K| \setminus \{v\})$ capture the local topological structure of the geometric realization $|K|$ near $v$.

Let $L$ be a subcomplex of $K$ containing the closed star of $v$. By the excision theorem, the inclusion $L \hookrightarrow K$ induces an isomorphism
\[
H_n(|L|, |L| \setminus \{v\}) \cong H_n(|K|, |K| \setminus \{v\})
\]
for $n\geq 0$.

For large complexes $K$, direct computation of these local homology groups can be computationally demanding. Therefore, one seeks a suitable subcomplex $L \subseteq K$ (such as the closed star of $v$) that contains $v$ and is small enough to simplify the calculation while ensuring the isomorphism holds.

For a vertex $v \in K$, we define the \emph{deleted subcomplex} $K \setminus \{v\}$ as the maximal subcomplex of $K$ disjoint from the open star $\mathrm{st}_K(v)$. There exists a natural isomorphism
\begin{equation*}
    H_n(|K|, |K| \setminus \mathrm{st}_K(v)) \cong H_n(K, K \setminus \{v\}), \quad \forall n \ge 0.
\end{equation*}
Since the topological pair $(|K|, |K| \setminus \mathrm{st}_K(v))$ is homotopy equivalent to $(|K|, |K| \setminus \{v\})$, we shall, for conciseness, adopt $H_n(K, K \setminus \{v\})$ as the formal notation for the \emph{local homology} of $K$ at $v$. Furthermore, this leads to the fundamental isomorphism
\begin{equation*}
    H_n(K, K \setminus \{v\}) \cong \tilde{H}_{n-1}(\mathrm{Lk}_K(v)),
\end{equation*}
which establishes a purely combinatorial bridge between local homology and reduced homology of the link of a vertex.

\subsubsection{Local homology and interaction homology}

Following the foundational work in \cite{knill2018cohomology,liu2023interaction}, interaction homology provides a robust invariant for capturing multi-body relationships within simplicial complexes, with significant applications in molecular biology and materials science \cite{chen2025interaction,liu2025persistent}.

Let $K$ and $L$ be simplicial complexes with associated chain complexes $C_\ast(K)$ and $C_\ast(L)$ over $\mathbb{K}$. Their tensor product $(\mathcal{C}_\ast, d)$ is a chain complex, with boundary given by the standard Leibniz rule. Consider the subcomplex $D_\ast(\{K, L\}) \subseteq \mathcal{C}_\ast$ generated by elementary tensors with disjoint support
\begin{equation*}
    S_{\emptyset} = \{ \sigma \otimes \tau \in \mathcal{C}_\ast \mid \sigma \cap \tau = \emptyset \}.
\end{equation*}
The \emph{interaction chain complex} is then given by the quotient complex
\begin{equation*}
    IC_\ast(\{K, L\}) = \mathcal{C}_\ast / D_\ast(\{K, L\}).
\end{equation*}

\begin{definition}
The \emph{interaction homology} of $\{K, L\}$ is defined as the homology of the interaction chain complex
\begin{equation*}
    H_n(\{K, L\}) = H_n(IC_\ast(\{K, L\})), \quad n \ge 0.
\end{equation*}
\end{definition}

The \emph{Wu characteristic} of $K$, as introduced in \cite{wen1959topologie}, is defined by $\omega(K) = \sum_{\sigma \cap \tau \neq \emptyset} (-1)^{\dim \sigma + \dim \tau}$. This invariant enumerates the weighted intersection patterns of all simplices within $K$.

\begin{theorem}[\cite{knill2018cohomology, liu2023interaction}]
For $K=L$, the interaction Betti numbers $\beta_{n} = \mathrm{rank}(H_n(K, K))$ provide a categorical refinement of the Wu characteristic, yielding the identity
\begin{equation*}
    \omega(K) = \sum_{n \ge 0} (-1)^n \beta_{n}.
\end{equation*}
\end{theorem}

Furthermore, there exists an intrinsic theoretical correspondence between interaction homology and local homology.

\begin{theorem}
Let $K$ be a simplicial complex and $v$ be a vertex in $K$. There is an isomorphism between the interaction homology and the local homology
\begin{equation*}
    H_n(\{K, \{v\}\}) \cong H_n(K, K \setminus \{v\}), \quad n \ge 0.
\end{equation*}
\end{theorem}

\begin{proof}
Consider the chain complex $C_\ast(\{v\})$ associated with the 0-simplex $\{v\}$. This complex is concentrated in degree 0, i.e., $C_0(\{v\}) \cong \mathbb{K} \cdot v$ and $C_i(\{v\}) = 0$ for $i > 0$. The tensor product complex $\mathcal{C}_n = \bigoplus_{i+j=n} C_i(K) \otimes C_j(\{v\})$ thus simplifies to
\begin{equation*}
    \mathcal{C}_n = C_n(K) \otimes C_0(\{v\}) \cong C_n(K).
\end{equation*}

Recall the definition of the subcomplex $D_n(\{K, \{v\}\})$, which is generated by elementary tensors with disjoint support
\begin{equation*}
    S_{\emptyset} = \{ \sigma \otimes v \in \mathcal{C}_n \mid \sigma \cap \{v\} = \emptyset \}.
\end{equation*}
The condition $\sigma \cap \{v\} = \emptyset$ is satisfied if and only if the vertex $v$ is not a face of $\sigma$, which means $\sigma \in K \setminus \{v\}$. Thus, the subcomplex $D_n$ is precisely the chain complex of the subcomplex $K \setminus \{v\}$, that is,
\begin{equation*}
    D_n(\{K, \{v\}\}) \cong C_n(K \setminus \{v\}).
\end{equation*}

The interaction chain complex is defined as the quotient
\begin{equation*}
    IC_n(\{K, \{v\}\}) = \mathcal{C}_n / D_n(\{K, \{v\}\}) \cong C_n(K) / C_n(K \setminus \{v\}).
\end{equation*}
The right-hand side is the definition of the relative chain complex $C_n(K, K \setminus \{v\})$. Since the boundary operator on $IC_\ast$ is induced by the tensor product differential (which coincides with the simplicial boundary operator $d_K$ in this degree-0 tensor case), we have an isomorphism of chain complexes
\begin{equation*}
    IC_\ast(\{K, \{v\}\}) \cong C_\ast(K, K \setminus \{v\}).
\end{equation*}
Taking the homology of both sides yields the desired result.
\end{proof}

\subsection{Local Laplacian}\label{section:local_Laplacian}

\subsubsection{Inner product structure and Hilbert space}

From now on, we restrict the coefficient field to $\mathbb{K} = \mathbb{R}$. Given that $K$ is a finite simplicial complex, each chain group $C_n(K)$ is a finite-dimensional vector space. We endow $C_n(K)$ with a canonical inner product by declaring the set of $n$-simplices $\{\sigma_i\}_{i \in I_n}$ to be an orthonormal basis. Formally, for any two $n$-simplices $\sigma_i, \sigma_j$, the inner product is defined as
\[
\langle \sigma_i, \sigma_j \rangle = \delta_{ij},
\]
where $\delta_{ij}$ is the Kronecker delta. By extending this bilinearly, for any chains $x = \sum_i a_i \sigma_i$ and $y = \sum_i b_i \sigma_i$ in $C_n(K)$, we have $\langle x, y \rangle = \sum_i a_i b_i$.

This construction equips $C_n(K)$ with the structure of a finite-dimensional \textit{Hilbert space}. Consequently, the entire chain complex $(C_\ast(K), d_\ast)$ can be viewed as a direct sum of Hilbert spaces.

\subsubsection{Combinatorial Laplacian}

The Hilbert structure on $C_\ast(K)$ allows for the definition of the \textit{coboundary operator} (or adjoint boundary operator) $d_n^\ast\colon C_{n-1}(K) \to C_n(K)$. This operator is uniquely determined by the adjoint relation
\[
\langle d_n x, y \rangle_{C_{n-1}} = \langle x, d_n^\ast y \rangle_{C_n}, \quad \forall x \in C_n(K), y \in C_{n-1}(K).
\]

Through the interplay of the boundary and coboundary operators, we recover the combinatorial analogue of the Laplace-Beltrami operator.

\begin{definition}
The \textit{$n$-th combinatorial Laplacian} $\Delta_n\colon C_n(K) \to C_n(K)$ is defined as the self-adjoint operator
\[
\Delta_n := d_{n+1} d_{n+1}^\ast + d_n^\ast d_n,\quad n\geq 0.
\]
For the base case $n=0$, the operator reduces to the graph Laplacian $\Delta_0 = d_1 d_1^\ast$, as $d_0$ is the zero map.
\end{definition}

The Laplacian $\Delta_n$ exhibits several fundamental properties essential for spectral analysis:
\begin{itemize}
    \item \textit{Self-adjointness}: $\langle \Delta_n x, y \rangle = \langle x, \Delta_n y \rangle$ for all $x, y \in C_n(K)$.
    \item \textit{Positive semi-definiteness}: $\langle \Delta_n x, x \rangle = \|d_{n+1}^\ast x\|^2 + \|d_n x\|^2 \geq 0$, implying a non-negative spectrum $\mathbf{Spec}(\Delta_n) \subset [0, \infty)$.
\end{itemize}

A cornerstone of combinatorial Hodge theory is the identification of the kernel of $\Delta_n$ with the topology of the complex. The elements in $\ker \Delta_n$ are referred to as \textit{combinatorial harmonic $n$-chains}.

\begin{theorem}[Simplicial Hodge theorem]
For a finite simplicial complex $K$, there exists a canonical isomorphism
\[
\ker \Delta_n \cong H_n(K; \mathbb{R}),\quad n\geq 0.
\]
\end{theorem}
The proof of the above result can be found in \cite{lim2020hodge,liu2024algebraic}.

\subsubsection{Inner product on the relative chain complex}

Let $(K, L)$ be a simplicial pair, where $L$ is a subcomplex of $K$. As $C_n(K)$ is a finite-dimensional Hilbert space, the subcomplex $C_n(L)$ constitutes a closed subspace. We define the inner product structure on the relative chain group $C_n(K, L) \cong C_n(K) / C_n(L)$ by invoking the orthogonal decomposition of the absolute chain groups. 

Specifically, let $C_n(L)^\perp$ denote the orthogonal complement of $C_n(L)$ within $C_n(K)$ with respect to the canonical inner product
\[
C_n(L)^\perp := \{ x \in C_n(K) \mid \langle x, y \rangle = 0, \, \forall y \in C_n(L) \}.
\]
The finite dimensionality of the complex ensures the direct sum decomposition $C_n(K) = C_n(L) \oplus C_n(L)^\perp$. Consequently, the quotient map $\pi_n\colon C_n(K) \to C_n(K, L)$ induces a canonical linear isomorphism between the orthogonal complement $C_n(L)^\perp$ and the relative chain group $C_n(K, L)$.

We endow $C_n(K, L)$ with the \textit{quotient inner product} such that $\theta_n=(\pi_n)|_{C_n(L)^\perp}\colon C_n(L)^\perp\to C_n(K, L)$ is an isometry. For any relative chains $\alpha, \beta \in C_n(K, L)$, the inner product is defined as
\[
\langle \alpha, \beta \rangle_{C_n(K, L)} := \langle x, y \rangle_{C_n(K)},
\]
where $x$ and $y$ are the unique representatives of $\alpha$ and $\beta$ in $C_n(L)^\perp$, respectively. Equivalently, this can be expressed as
\[
\langle [x], [y] \rangle = \langle \theta_n(x), \theta_n(y) \rangle,
\]
which is independent of the choice of representatives in $C_n(K)$. 

Under this metric structure, $C_\ast(K, L)$ becomes a chain complex of Hilbert spaces. This construction ensures that the relative boundary operator $\bar{d}_n\colon C_n(K, L) \to C_{n-1}(K, L)$ admits a well-defined adjoint.

\subsubsection{The relative combinatorial Laplacian}

To extend the spectral theory to the relative setting, we consider the induced operators on the quotient Hilbert complex $(C_\ast(K, L), \bar{d}_\ast)$. The \textit{relative boundary operator} $\bar{d}_n\colon C_n(K, L) \to C_{n-1}(K, L)$ is the unique homomorphism making the following diagram commute:
\[
\bar{d}_n \circ \pi_n = \pi_{n-1} \circ d_n,
\]
where $\pi_n\colon C_n(K) \to C_n(K, L)$ is the canonical quotient map. The well-definedness of $\bar{d}_n$ is guaranteed by the subcomplex condition $d_n(C_n(L)) \subseteq C_{n-1}(L)$.

\begin{definition}
The \textit{relative coboundary operator} $\bar{d}_n^{\ast}\colon C_{n-1}(K, L) \to C_n(K, L)$ is the Hilbert adjoint of $\bar{d}_n$ with respect to the quotient inner product, defined by the identity
\[
\langle \bar{d}_n \alpha, \beta \rangle_{C_{n-1}(K, L)} = \langle \alpha, \bar{d}_n^{\ast} \beta \rangle_{C_n(K, L)}, \quad \forall \alpha \in C_n(K, L), \beta \in C_{n-1}(K, L).
\]
\end{definition}

In practical computations, $\bar{d}_n^{\ast}$ can be realized via the orthogonal projection $P_n\colon C_n(K) \to C_n(L)^\perp$. Given a relative chain $\beta \in C_{n-1}(K, L)$, let $d \in C_{n-1}(L)^\perp$ be its unique isometric representative. The action of the relative coboundary operator is then given by
\[
\bar{d}_n^{\ast} \beta = \pi_n \left( P_n (d_n^{\ast} d) \right),
\]
where $d_n^{\ast}$ on the right-hand side denotes the adjoint operator in the absolute chain complex $C_\ast(K)$. The projection $P_n$ explicitly enforces the condition that the resulting gradient-like flow resides within the orthogonal complement of the subcomplex $L$. 

One can define the relative combinatorial Laplacian on a simplicial complex \cite{zhan2025combinatorial}.

\begin{definition}
The \textit{$n$-th relative combinatorial Laplacian} $\Delta_n^{K,L}\colon C_n(K, L) \to C_n(K, L)$ is the self-adjoint operator defined as
\[
\Delta_n^{K,L} := \bar{d}_{n+1} \bar{d}_{n+1}^{\ast} + \bar{d}_n^{\ast} \bar{d}_n.
\]
In the base case $n=0$, the operator simplifies to $\Delta_0^{K,L} = \bar{d}_{1} \bar{d}_{1}^{\ast}$.
\end{definition}

The relative Laplacian $\Delta_n^{K,L}$ inherits the essential spectral properties of its absolute counterpart: it is self-adjoint and positive semi-definite. Furthermore, it provides a spectral decomposition of the relative homology groups. The space of \textit{harmonic relative $n$-chains} is defined as $\mathcal{H}_n(K, L) = \ker \Delta_n^{K,L}$.

\begin{theorem}
Let $(K, L)$ be a finite simplicial pair and $C_n(K, L)$ be the $n$-th relative chain group equipped with the quotient inner product. There exists a canonical orthogonal decomposition of the relative chain space into orthogonal subspaces
\[
C_n(K, L) = \im \bar{d}_{n+1} \oplus \im \bar{d}_n^\ast \oplus \mathcal{H}_n(K, L),
\]
where $\bar{d}_{n+1}$ and $\bar{d}_n^\ast$ are the relative boundary and coboundary operators, respectively. Moreover, we have an canonical isomorphism
\[
\mathcal{H}_n(K, L)\cong H_n(K, L; \mathbb{R}).
\]
\end{theorem}

\begin{proof}
The proof follows the parallelly to the classical (absolute) Hodge decomposition by restricting the analysis to the space $C_n(K, L)$.
\end{proof}

\subsubsection{Local Laplacian}

For a vertex $v \in K$, we define the deleted subcomplex $L = K \setminus \{v\}$ as the maximal subcomplex of $K$ disjoint from the open star of $v$. The associated relative chain complex $C_\ast(K, K \setminus \{v\})$ thus provides an algebraic encoding of the local topology in the immediate vicinity of $v$.

\begin{definition}
The \textit{$n$-th local combinatorial Laplacian} at vertex $v$, denoted by $\Delta_n^{K,v}\colon C_n(K, K \setminus \{v\}) \to C_n(K, K \setminus \{v\})$, is defined as the self-adjoint operator
\[
\Delta_n^{K,v} := \bar{d}_{n+1} \bar{d}_{n+1}^{\ast} + \bar{d}_n^{\ast} \bar{d}_n,
\]
where $\bar{d}$ and $\bar{d}^\ast$ represent the boundary and coboundary operators induced on the quotient Hilbert space, respectively. In the base case, we define $\Delta_0^{K,v} = \bar{d}_{1} \bar{d}_{1}^{\ast}$.
\end{definition}

The local Laplacian is a canonical specialization of the relative Laplacian, specifically designed to characterize the topological invariants and spectral density localized at a singular point. 
Moreover, $\Delta_n^{K,v}$ is a self-adjoint, positive semi-definite operator whose spectral properties are intrinsically linked to the local geometry. Specifically, the kernel of this operator isomorphic to the local homology group $H_n(K, K \setminus \{v\})$.
\begin{theorem}
Let $K$ be a finite simplicial complex and $v \in K$ a vertex. We have an isomorphism of vector spaces
\[
\ker \Delta_n^{K,v} \cong H_n(K, K \setminus \{v\}; \mathbb{R}).
\]
Furthermore, each local homology class $[\alpha] \in H_n(K, K \setminus \{v\}; \mathbb{R})$ contains a unique harmonic representative $\gamma \in C_n(K, K \setminus \{v\})$ such that $\Delta_n^{K,v} \gamma = 0$.
\end{theorem}

The non-zero spectrum encapsulates the combinatorial curvature and the local connectivity structure, effectively isolating the relative contribution of the vertex to the global homological landscape. This construction facilitates a rigorous localized harmonic analysis, enabling the detection of local singularities and manifold-like configurations directly from the simplicial complexes.

\subsection{Computational local Laplacian}

\subsubsection{Computing local Laplacians}

Let $K$ be a simplicial complex, and let $v$ be a vertex in $K$.
Define the map $\phi\colon C_n(K, K \setminus \{v\}) \to C_{n-1}(\text{Lk}_K(v))$ by
\[
\phi([\tau]) = (-1)^n [v_1, \dots, v_n],\quad n\geq 1,
\]
where $\tau = [v, v_1, \dots, v_n]$ and $[v_1, \dots, v_n] \in \text{Lk}_K(v)$ is the induced orientation.

\begin{theorem}\label{theorem:chain_isomorphism}
Let $K$ be a simplicial complex and let $v \in K$ be a vertex. Then there exists an isomorphism of chain complexes
\[
\phi\colon C_{\ast}(K, K \setminus \{v\}) \;\longrightarrow\; C_{\ast-1}(\Lk_{K}(v)).
\]
Moreover, this chain isomorphism is an isometry with respect to the quotient inner product on $C_{\ast}(K, K \setminus \{v\})$ and the standard inner product on $C_{\ast-1}(\Lk_{K}(v))$.
\end{theorem}

\begin{proof}
The chain isomorphism is established in \cite{munkres2018elements}.

Now, we verify that it preserves the quotient inner product on $C_n(K, K \setminus \{v\})$ with respect to the standard inner product on $C_{n-1}(\text{Lk}_K(v))$.

For $[x], [y] \in C_n(K, K \setminus \{v\})$, consider the representatives $x' = \sum\limits_{\tau = \{v, v_1, \dots, v_n\}} x'(\tau) \tau$ and $y' = \sum\limits_{\tau} y'(\tau) \tau \in C_n(K \setminus \{v\})^\perp$ of $[x]$ and $[y]$, respectively. Applying $\phi$, we have
\[
\phi([x]) = \sum_{\tau = [v, v_1, \dots, v_n]} x'(\tau) (-1)^n [v_1, \dots, v_n] = \sum_{\sigma \in \text{Lk}_K(v)} x'(\{v\} \cup \sigma) (-1)^n \sigma,
\]
\[
\phi([y]) = \sum_{\sigma \in \text{Lk}_K(v)} y'(\{v\} \cup \sigma) (-1)^n \sigma.
\]
Compute the inner product in $C_{n-1}(\text{Lk}_K(v))$, one has
\begin{align*}
  \langle \phi([x]), \phi([y]) \rangle_{C_{n-1}(\text{Lk}_K(v))}= & \sum_{\sigma \in \text{Lk}_K(v)} \left( x'(\{v\} \cup \sigma) (-1)^n \right) \left( y'(\{v\} \cup \sigma) (-1)^n \right)  \\
  = & \sum_{\sigma} x'(\{v\} \cup \sigma) y'(\{v\} \cup \sigma) (-1)^{2n}.
\end{align*}
It follows that
\[
\langle \phi([x]), \phi([y]) \rangle_{C_{n-1}(\text{Lk}_K(v))} = \sum_{\sigma \in \text{Lk}_K(v)} x'(\{v\} \cup \sigma) y'(\{v\} \cup \sigma) = \sum_{\tau = \{v, v_1, \dots, v_n\}} x'(\tau) y'(\tau).
\]
This matches the quotient inner product
\[
\langle [x], [y] \rangle_{C_n(K, K \setminus \{v\})} = \sum_{\tau = \{v, v_1, \dots, v_n\}} x'(\tau) y'(\tau).
\]
Thus, we obtain
\[
\langle [x], [y] \rangle_{C_n(K, K \setminus \{v\})} = \langle \phi([x]), \phi([y]) \rangle_{C_{n-1}(\text{Lk}_K(v))}.
\]
Hence, $\phi$ is an isometry.
\end{proof}

By Theorem~\ref{theorem:chain_isomorphism}, the local Laplacian on the relative chain group $C_{n}(K,\,K\setminus\{v\})$ can be computed via the combinatorial Laplacian on the link
complex $C_{\ast}(\Lk)$. More precisely, let
\[
\Delta_{n}^{K,v}\colon C_{n}(K,\,K\setminus\{v\}) \longrightarrow C_{n}(K,\,K\setminus\{v\})
\]
denote the combinatorial Laplacian defined with respect to the quotient inner product, and let
\[
\Delta_{n-1}^{\Lk}\colon C_{n-1}(\Lk) \longrightarrow C_{n-1}(\Lk)
\]
denote the standard combinatorial Laplacian on the chain complex of the link. Then we have the following result.

\begin{theorem}\label{theorem:laplacian_conjugacy}
Let $K$ be a simplicial complex and $v \in K$ a vertex. Then the following diagram
\[
  \xymatrix{
  C_n(K,\,K\setminus\{v\}) \ar[rr]^-{\Delta_{n}^{K,v}} \ar[d]_{\phi}^{\cong}
    && C_n(K,\,K\setminus\{v\}) \ar[d]^{\phi}_{\cong} \\
  C_{n-1}(\Lk) \ar[rr]^-{\Delta_{n-1}^{\Lk}}
    && C_{n-1}(\Lk)
  }
\]
commutes for $n\geq 1$. In particular, we have
\[
\Delta_{n}^{K,v} \;=\; \phi^{-1} \circ \Delta_{n-1}^{\Lk} \circ \phi,\quad n\geq 1.
\]
\end{theorem}

\begin{proof}
Given that $\phi\colon C_n(K, K \setminus \{v\}) \to C_{n-1}(\Lk)$ is a chain isomorphism and an isometry, we show that $\phi \circ \Delta_n^{K,v} = \Delta_{n-1}^{\Lk} \circ \phi$.

The local Laplacian is defined as $\Delta_n^{K,v} = d_{n+1} d_{n+1}^{\ast} + d_n^{\ast} d_n$ on $C_n(K, K \setminus \{v\})$, and the link Laplacian is $\Delta_{n-1}^{\Lk} = d_n^{\mathrm{Lk}} (d_n^{\mathrm{Lk}})^{\ast} + (d_{n-1}^{\mathrm{Lk}})^{\ast} d_{n-1}^{\mathrm{Lk}}$ on $C_{n-1}(\Lk)$. Since $\phi$ is a chain map, $\phi \circ d_n = d_{n-1}^{\mathrm{Lk}} \circ \phi$ and $\phi \circ d_{n+1} = d_n^{\mathrm{Lk}} \circ \phi$.

To show commutativity, we verify for each term. For the first term, we have
\[
\phi \circ (d_{n+1} d_{n+1}^{\ast}) = (\phi \circ d_{n+1}) \circ d_{n+1}^{\ast} = (d_n^{\mathrm{Lk}} \circ \phi) \circ d_{n+1}^{\ast}.
\]
Since $\phi$ is an isometry, for $[a] \in C_{n+1}(K, K \setminus \{v\})$, $[b] \in C_n(K, K \setminus \{v\})$,
\[
\langle d_{n+1} [a], [b] \rangle_{C_n(K, K \setminus \{v\})} = \langle \phi(d_{n+1} [a]), \phi([b]) \rangle_{C_{n-1}(\Lk)} = \langle d_n^{\mathrm{Lk}} \phi([a]), \phi([b]) \rangle = \langle \phi([a]), (d_n^{\mathrm{Lk}})^{\ast} \phi([b]) \rangle.
\]
Also, $\langle d_{n+1} [a], [b] \rangle = \langle [a], d_{n+1}^{\ast} [b] \rangle = \langle \phi([a]), \phi(d_{n+1}^{\ast} [b]) \rangle$. Thus, $\phi \circ d_{n+1}^{\ast} = (d_n^{\mathrm{Lk}})^{\ast} \circ \phi$. Hence,
\[
\phi \circ (d_{n+1} d_{n+1}^{\ast}) = d_n^{\mathrm{Lk}} \circ (d_n^{\mathrm{Lk}})^{\ast} \circ \phi.
\]
For the second term, we obtain
\[
\phi \circ (d_n^{\ast} d_n) = (\phi \circ d_n^{\ast}) \circ d_n = ((d_{n-1}^{\mathrm{Lk}})^{\ast} \circ \phi) \circ d_n = (d_{n-1}^{\mathrm{Lk}})^{\ast} \circ (\phi \circ d_n) = (d_{n-1}^{\mathrm{Lk}})^{\ast} d_{n-1}^{\mathrm{Lk}} \circ \phi,
\]
since $\phi \circ d_n^{\ast} = (d_{n-1}^{\mathrm{Lk}})^{\ast} \circ \phi$ by a similar adjoint argument. Combining both,
\[
\phi \circ \Delta_n^{K,v} = \phi \circ (d_{n+1} d_{n+1}^{\ast} + d_n^{\ast} d_n) = d_n^{\mathrm{Lk}} (d_n^{\mathrm{Lk}})^{\ast} \circ \phi + (d_{n-1}^{\mathrm{Lk}})^{\ast} d_{n-1}^{\mathrm{Lk}} \circ \phi = \Delta_{n-1}^{\Lk} \circ \phi.
\]
Thus, the diagram commutes, and $\Delta_n^{K,v} = \phi^{-1} \circ \Delta_{n-1}^{\Lk} \circ \phi$.
\end{proof}

Theorem~\ref{theorem:laplacian_conjugacy} implies that the matrix representation of the Laplacian $\Delta_{n}^{K,v}$ is conjugate to the matrix representation of the Laplacian $\Delta_{n-1}^{\Lk}$ via the change of basis induced by the chain isomorphism $\phi$, and consequently the operators $\Delta_{n}^{K,v}$ and $\Delta_{n-1}^{\Lk}$ have the same spectrum. In this sense, we may regard $\Delta_{n-1}^{\Lk}$ as the local Laplacian, differing from $\Delta_{n}^{K,v}$ by a dimension shift of $-1$.

\begin{proposition}\label{proposition:zero_Laplacian}
Let $K$ be a simplicial complex and $v \in K$ a vertex. Then we have
\begin{equation*}
  \Delta_{0}^{K,v}=\deg v,
\end{equation*}
where $\deg v$ denotes the number of edges connecting $v$.
\end{proposition}

\begin{proof}
The relative chain group $C_0(K, K \setminus \{v\}) = C_0(K) / C_0(K \setminus \{v\})$ is generated by the class $[v]$, since $C_0(K)$ is generated by all vertices of $K$, and $C_0(K \setminus \{v\})$ by vertices other than $v$. Thus, $C_0(K, K \setminus \{v\})$ is 1-dimensional, spanned by $[v]$.

Next, consider $d_1\colon C_1(K, K \setminus \{v\}) \to C_0(K, K \setminus \{v\})$. The chain group $C_1(K, K \setminus \{v\}) = C_1(K) / C_1(K \setminus \{v\})$ is generated by classes of edges $[v, w]$ where $\{v, w\} \in K$. The boundary map is
\[
d_1([v, w]) = [w] - [v]= 0 - [v] = -[v].
\]
Here, we use that fact $[w] = 0$ in $C_0(K, K \setminus \{v\})$. Thus, $d_1$ maps each basis element $[v, w]$ to a multiple of $[v]$.

Suppose $d_1^{\ast} [v] = \sum_{w\colon \{v, w\} \in K} c_w [v, w]$. Then we have
\[
c_{w'} = \langle [v, w'], d_1^{\ast} [v] \rangle_{C_1} = \langle d_1 [v, w'], [v] \rangle_{C_0} = -1.
\]
It follows that $d_1^{\ast} [v] = \sum_{w\colon \{v, w\} \in K} -[v, w]$. Thus, one has
\[
d_1 d_1^{\ast} [v] = d_1 \left( \sum_{w\colon \{v, w\} \in K} (-1) [v, w] \right) = \sum_{w} (-1) d_1 [v, w] = \sum_{w} (-1) (-[v]) = (\deg v) [v],
\]
where $\deg v$ is the number of edges incident to $v$. Hence, $d_1 d_1^{\ast} = (\deg v) \cdot \text{id}$.
\end{proof}
It is worth noting that even if the simplicial complex $K$ is very large, in the sense that its vertex set and simplex set are extensive, the local subcomplex $\Lk_K(v)$ may still be comparatively small, which can lead to significant improvements in the efficiency of computing local information.

\subsubsection{Examples}

\begin{example}\label{example:local_lap}
Consider a simplicial complex $K$ with vertices $\{v, a, b, c\}$, edges
\begin{equation*}
  \{v, a\}, \{v, b\}, \{v, c\}, \{a, b\}, \{b, c\}
\end{equation*}
and triangles $\{v, a, b\}, \{v, b, c\}$. We compute the local Laplacians $\Delta_1^{K,v}$ and $\Delta_2^{K,v}$ for vertex $v$ using the Laplacians of the link $\mathrm{Lk}_K(v)$.

The link $\mathrm{Lk}_K(v)$ consists of vertices $\{a, b, c\}$ and edges $\{a, b\}, \{b, c\}$. The chain group $C_0(\mathrm{Lk}_K(v))$ is generated by $\{[a], [b], [c]\}$, $C_1(\mathrm{Lk}_K(v))$ by $\{[a, b], [b, c]\}$, and higher chains are zero.

\begin{figure}[htbp]
    \centering
    \begin{tikzpicture}[scale=1.6, thick,
        node/.style={circle, fill=black, inner sep=1.5pt},
        vnode/.style={circle, fill=red, inner sep=2pt, label=left:$v$}]

        \draw[dashed, gray!60] (-0.5, -1.2) rectangle (2.5, 1.2);
        \begin{scope}
            \fill[blue!15] (0,0) -- (1, 0.8) -- (1, -0.8) -- cycle; 
            \fill[green!15] (0,0) -- (1, -0.8) -- (2, 0) -- cycle; 
            \node[vnode] (v) at (0,0) {};
            \node[node, label=above:$a$] (a) at (1, 0.8) {};
            \node[node, label=below:$b$] (b) at (1, -0.8) {};
            \node[node, label=right:$c$] (c) at (2, 0) {};
            \draw (v) -- (a); \draw (v) -- (b); \draw (v) -- (c);
            \draw (a) -- (b); \draw (b) -- (c);
            \node at (2, -1.0) {$K$};
        \end{scope}

        \draw[->, -Stealth, thick] (2.7, 0) -- (3.5, 0) node[midway, above] {$\mathrm{Lk}_K(v)$};

        \begin{scope}[xshift=4.2cm]
            \draw[dashed, gray!60] (-0.5, -1.2) rectangle (1.5, 1.2);
            \node[node, label=above:$a$] (la) at (0, 0.8) {};
            \node[node, label=below:$b$] (lb) at (0, -0.8) {};
            \node[node, label=right:$c$] (lc) at (1, 0) {};
            \draw[blue, thick] (la) -- (lb);
            \draw[green!60!black, thick] (lb) -- (lc);
            \node at (1, -1.0) {$\mathrm{Lk}_K(v)$};
        \end{scope}
    \end{tikzpicture}
    \caption{Construction of the link complex from a simplicial complex $K$.}
\end{figure}
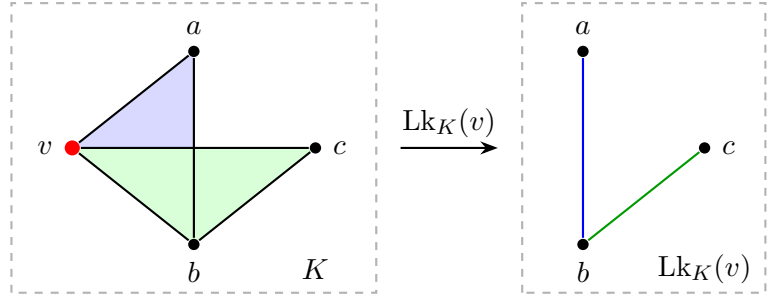

For $\Delta_1^{K,v}$ on $C_1(K, K \setminus \{v\})$ generated by $\{[v, a], [v, b], [v, c]\}$: the Laplacian is $\Delta_0^{\mathrm{Lk}_K(v)} = d_1^{\mathrm{Lk}} (d_1^{\mathrm{Lk}})^{\ast} + (d_0^{\mathrm{Lk}})^{\ast} d_0^{\mathrm{Lk}}$. Since $d_0^{\mathrm{Lk}} = 0$, we have $\Delta_0^{\mathrm{Lk}_K(v)} = d_1^{\mathrm{Lk}} (d_1^{\mathrm{Lk}})^{\ast}$.

The boundary operator is
\[
d_1^{\mathrm{Lk}}([a, b]) = [b] - [a], \quad d_1^{\mathrm{Lk}}([b, c]) = [c] - [b].
\]
The corresponding presentation matrix with respect to the basis $\{[a,b], [b,c]\}$ in the domain and the basis $\{[a], [b], [c]\}$ in the codomain is
\[
B_1^{\mathrm{Lk}} = \begin{bmatrix}
-1 & 0 \\
1 & -1 \\
0 & 1
\end{bmatrix}, \quad (B_1^{\mathrm{Lk}})^{\ast} = \begin{bmatrix}
-1 & 1 & 0 \\
0 & -1 & 1
\end{bmatrix}.
\]
Thus, we obtain the representation matrix of Laplacian
\[
L_0^{\mathrm{Lk}_K(v)} = \begin{bmatrix}
-1 & 0 \\
1 & -1 \\
0 & 1
\end{bmatrix} \begin{bmatrix}
-1 & 1 & 0 \\
0 & -1 & 1
\end{bmatrix} = \begin{bmatrix}
1 & -1 & 0 \\
-1 & 2 & -1 \\
0 & -1 & 1
\end{bmatrix}.
\]
By Theorem \ref{theorem:laplacian_conjugacy}, the representation matrix of $\Delta_1^{K,v}$ in the corresponding basis is
\[
L_1^{K,v} = \begin{bmatrix}
1 & -1 & 0 \\
-1 & 2 & -1 \\
0 & -1 & 1
\end{bmatrix}.
\]

For $\Delta_2^{K,v}$ on $C_2(K, K \setminus \{v\})$ generated by $\{[v, a, b], [v, b, c]\}$: the Laplacian is $\Delta_1^{\mathrm{Lk}_K(v)} = d_2^{\mathrm{Lk}} (d_2^{\mathrm{Lk}})^{\ast} + (d_1^{\mathrm{Lk}})^{\ast} d_1^{\mathrm{Lk}}$. Since $d_2^{\mathrm{Lk}} = 0$, we have $\Delta_1^{\mathrm{Lk}_K(v)} = (d_1^{\mathrm{Lk}})^{\ast} d_1^{\mathrm{Lk}}$. Thus, we have
\[
\Delta_1^{\mathrm{Lk}_K(v)} = (B_1^{\mathrm{Lk}})^{\ast} B_1^{\mathrm{Lk}} = \begin{bmatrix}
-1 & 1 & 0 \\
0 & -1 & 1
\end{bmatrix} \begin{bmatrix}
-1 & 0 \\
1 & -1 \\
0 & 1
\end{bmatrix} = \begin{bmatrix}
2 & -1 \\
-1 & 2
\end{bmatrix}.
\]
The representation matrix of $\Delta_2^{K,v}$ with respect to the corresponding basis is
\[
L_2^{K,v} = \begin{bmatrix}
2 & -1 \\
-1 & 2
\end{bmatrix}.
\]
\end{example}

\begin{example}
For the same simplicial complex $K$ as in Example~\ref{example:local_lap}, we compute the local Laplacians at vertex $a$.

The link $\mathrm{Lk}_K(a)$ consists of vertices $\{v,b\}$ and the single edge $\{v,b\}$. Thus
\[
C_0(\mathrm{Lk}_K(a)) = \langle [v], [b] \rangle,
\qquad
C_1(\mathrm{Lk}_K(a)) = \langle [v,b] \rangle,
\]
with boundary operator
\[
d_1^{\mathrm{Lk}}([v,b]) = [b] - [v].
\]

For $\Delta_1^{K,a}$ on $C_1(K,K\setminus \{a\})$ generated by $\{[a,v],[a,b]\}$: since $d_0^{\mathrm{Lk}}=0$, we have
\[
\Delta_0^{\mathrm{Lk}_K(a)} = d_1^{\mathrm{Lk}} (d_1^{\mathrm{Lk}})^{\ast}.
\]
With respect to the basis $\{[v,b]\}$ in the domain and $\{[v],[b]\}$ in the codomain,
\[
B_1^{\mathrm{Lk}} = \begin{bmatrix}-1 \\ 1\end{bmatrix},
\qquad
(B_1^{\mathrm{Lk}})^{\ast} = \begin{bmatrix}-1 & 1\end{bmatrix}.
\]
Thus
\[
L_0^{\mathrm{Lk}_K(a)} = B_1^{\mathrm{Lk}} (B_1^{\mathrm{Lk}})^{\ast}
= \begin{bmatrix}-1 \\ 1\end{bmatrix}\begin{bmatrix}-1 & 1\end{bmatrix}
= \begin{bmatrix} 1 & -1 \\ -1 & 1 \end{bmatrix}.
\]
By Theorem~\ref{theorem:laplacian_conjugacy}, the representation matrix of $\Delta_1^{K,a}$ is
\[
L_1^{K,a} = \begin{bmatrix} 1 & -1 \\ -1 & 1 \end{bmatrix},
\]
with respect to the basis $\{[a,v],[a,b]\}$.

For $\Delta_2^{K,a}$ on $C_2(K,K\setminus \{a\})$ generated by $\{[v,a,b]\}$: since $d_2^{\mathrm{Lk}}=0$, we have
\[
\Delta_1^{\mathrm{Lk}_K(a)} = (d_1^{\mathrm{Lk}})^{\ast} d_1^{\mathrm{Lk}}
= (B_1^{\mathrm{Lk}})^{\ast} B_1^{\mathrm{Lk}}.
\]
Hence
\[
L_2^{K,a} = (B_1^{\mathrm{Lk}})^{\ast} B_1^{\mathrm{Lk}}
= \begin{bmatrix}-1 & 1\end{bmatrix}\begin{bmatrix}-1 \\ 1\end{bmatrix}
= \begin{bmatrix} 2 \end{bmatrix}.
\]

In summary, the local Laplacians at vertex $a$ are
\[
L_1^{K,a} = \begin{bmatrix}1 & -1 \\ -1 & 1\end{bmatrix},
\qquad
L_2^{K,a} = \begin{bmatrix}2\end{bmatrix}.
\]

This example shows that the local Laplacians at different vertices may differ.
\end{example}

\begin{example}
Consider an $n$-simplex $K$ with the vertex set $\{0, 1, \dots, n\}$. The link of the vertex $0$, denoted as $\mathrm{Lk}_K(0)$, is an $(n-1)$-simplex spanned by the vertices $\{1, 2, \dots, n\}$. Consequently, the 1-skeleton of $\mathrm{Lk}_K(0)$ is the complete graph $K_n$ on $n$ vertices.
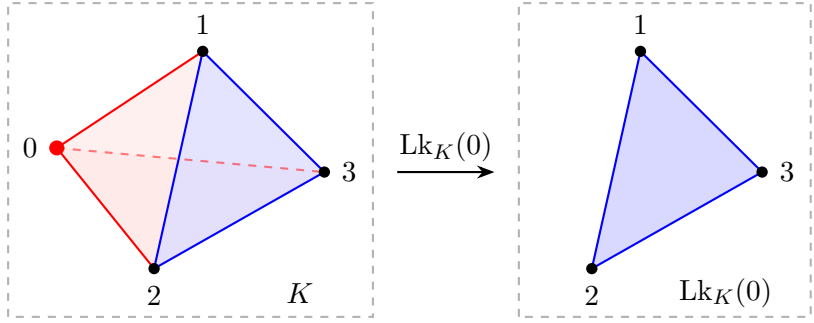
\begin{figure}[htbp]
    \centering
    \begin{tikzpicture}[scale=1.6, thick,
        node/.style={circle, fill=black, inner sep=1.5pt},
        vzero/.style={circle, fill=red, inner sep=2pt, label=left:$0$}]

        \begin{scope}
            \draw[dashed, gray!60] (-0.4, -1.4) rectangle (2.6, 1.2);
            
            \coordinate (0) at (0,0);
            \coordinate (1) at (1.2, 0.8);
            \coordinate (2) at (0.8, -1);
            \coordinate (3) at (2.2, -0.2);

            \fill[red!10, opacity=0.6] (0) -- (1) -- (2) -- cycle;
            \fill[red!10, opacity=0.6] (0) -- (2) -- (3) -- cycle;
            \fill[blue!15, opacity=0.7] (1) -- (2) -- (3) -- cycle; 

            \draw[red, dashed, opacity=0.5] (0) -- (3); 
            \draw[red] (0) -- (1);
            \draw[red] (0) -- (2);

            \draw[blue, thick] (1) -- (2);
            \draw[blue, thick] (2) -- (3);
            \draw[blue, thick] (1) -- (3);

            \node[vzero] at (0) {};
            \node[node, label=above:$1$] at (1) {};
            \node[node, label=below:$2$] at (2) {};
            \node[node, label=right:$3$] at (3) {};

            \node at (2, -1.2) {$K$};
        \end{scope}

        \draw[->, -Stealth, thick] (2.8, -0.2) -- (3.6, -0.2) node[midway, above] {$\mathrm{Lk}_K(0)$};

        \begin{scope}[xshift=4.6cm]
            \draw[dashed, gray!60] (-0.8, -1.4) rectangle (1.6, 1.2);
            
            \coordinate (L1) at (0.2, 0.8);
            \coordinate (L2) at (-0.2, -1);
            \coordinate (L3) at (1.2, -0.2);

            \fill[blue!15] (L1) -- (L2) -- (L3) -- cycle;

            \draw[blue, thick] (L1) -- (L2);
            \draw[blue, thick] (L2) -- (L3);
            \draw[blue, thick] (L1) -- (L3);

            \node[node, label=above:$1$] at (L1) {};
            \node[node, label=below:$2$] at (L2) {};
            \node[node, label=right:$3$] at (L3) {};

            \node at (0.9, -1.2) {$\mathrm{Lk}_K(0)$};
        \end{scope}

    \end{tikzpicture}
    \caption{The local structure of a $3$-simplex at vertex $0$. The link $\mathrm{Lk}_K(0)$ is a $2$-simplex (triangle), whose combinatorial Laplacian $\Delta_0^{\mathrm{Lk}_K(0)}$ defines the local Laplacian $\Delta_1^{K,0}$ of the original complex at vertex $0$.}
    \label{fig:n_simplex_example}
\end{figure}
The combinatorial Laplacian $\Delta_0^{\mathrm{Lk}_K(0)}$ acting on the $0$-th chain group $C_0(\mathrm{Lk}_K(0))$ corresponds to the standard graph Laplacian of $K_n$. Specifically, this operator is represented by an $n \times n$ matrix with diagonal entries $n-1$ and off-diagonal entries $-1$, which can be written as
$$\Delta_0^{\mathrm{Lk}_K(0)} = nI_n - J_n,$$
where $I_n$ denotes the $n \times n$ identity matrix and $J_n$ is the $n \times n$ all-ones matrix. 

For instance, in the case of $n=3$ (where the link is a $2$-simplex with vertices $\{1, 2, 3\}$), the matrix takes the form
\begin{equation*}
  \Delta_0^{\mathrm{Lk}_K(0)} = \begin{bmatrix*}[r]
 2 & -1 & -1 \\
-1 &  2 & -1 \\
-1 & -1 &  2
\end{bmatrix*}
\end{equation*}
By the chain isomorphism $\phi \colon C_1(K, K \setminus \{0\}) \to C_0(\mathrm{Lk}_K(0))$, we can define the local Laplacian as $\Delta_1^{K,0} = \phi^{-1} \circ \Delta_0^{\mathrm{Lk}_K(0)} \circ \phi$. Under the basis induced by $\phi$, $\Delta_1^{K,0}$ inherits the same matrix representation. Notably, due to the inherent vertex symmetry of the simplex, the local Laplacian is the same for every vertex.
\end{example}

\section{Topological persistence and local Laplacian}\label{section:generalized_Laplacian}

While persistent homology characterizes the evolution of global homological features across a filtration, local homology identifies the topological invariants restricted to the neighborhood of a specific point relative to its boundary. Persistent local homology integrates these paradigms by tracking the persistence of local homological features across multiple scales, thereby encoding scale-dependent geometric and topological signatures. 

Parallel to this development, we introduce the \textit{persistent local Laplacian}, which shifts the focus from purely local topological persistence to local spectral persistence. It provides a multiscale characterization of localized harmonic structures and spectral features.

\subsection{Persistent local homology}

\subsubsection{Filtration and persistent homology}

Let $\mathcal{K}$ be a \textit{persistent simplicial complex} (or a filtration) indexed by the integers $(\mathbb{Z}, \leq)$. Formally, this is a functor $\mathcal{K}\colon (\mathbb{Z}, \leq) \to \mathbf{Simp}$ from the discrete category of integers to the category of simplicial complexes, which can be visualized via the following long sequence
\begin{equation*}
\xymatrix{
\cdots \ar[r] \ar[r] & \mathcal{K}_{i-1} \ar[r]^{f^{i-1,i}} & \mathcal{K}_{i} \ar[r]^{f^{i,i+1}} & \mathcal{K}_{i+1} \ar[r]  & \cdots.
}
\end{equation*}
For any pair $i \leq j$, the simplicial map $f^{i,j}\colon \mathcal{K}_i \rightarrow \mathcal{K}_j$ induces a map of homology groups
\begin{equation*}
f^{i,j}_{n}\colon H_{n}(\mathcal{K}_{i}) \to H_{n}(\mathcal{K}_{j}), \quad \text{for } n \geq 0.
\end{equation*}

The persistent homological information is captured by the images of these induced maps, which represent the topological features that survive from scale $i$ to scale $j$.

\begin{definition}
For $i \leq j$, the \textit{$(i,j)$-persistent $n$-th homology group} of the filtration $\mathcal{K}$, denoted by $H_n^{i,j}(\mathcal{K})$, is defined as
\[
H_n^{i,j}(\mathcal{K}) := \im(f^{i,j}_{n}\colon H_n(\mathcal{K}_i) \to H_n(\mathcal{K}_j)).
\]
\end{definition}
The rank of this group, $\beta_n^{i,j} = \dim H_n^{i,j}(\mathcal{K})$, constitutes the \textit{$(i,j)$-persistent Betti number}, which counts the number of $n$-dimensional holes that are born at or before index $i$ and remain alive at index $j$.

Persistent homology allows for a multi-scale analysis of topological stability. Features with a large persistence $j-i$ are typically associated with the underlying geometry of the space, while those with short lifespans are often interpreted as noise inherent to the discretization or the filtration process.

\subsubsection{Persistent relative homology}

The concept of persistent local homology was pioneered by Fasy et al. \cite{fasy2016exploring}, with subsequent research extending these localized invariants to graph-structured data and geometric deep learning architectures \cite{cesa2023algebraic, wang2024persistent}. In this section, we provide a formal construction of persistent local homology by treating persistence complexes as functors.

Let $\mathcal{K}\colon (\mathbb{Z}, \leq) \to \mathbf{Simp}$ be a \textit{persistence simplicial complex}. A \textit{persistence subcomplex} $\mathcal{L}$ of $\mathcal{K}$ is defined as a subfunctor $\mathcal{L} \subseteq \mathcal{K}$. Specifically, for each $i \in \mathbb{Z}$, $\mathcal{L}_i$ is a subcomplex of $\mathcal{K}_i$, and for every $i \leq j$, the following diagram of simplicial maps commutes:
\begin{equation*}
\xymatrix{
\mathcal{L}_i \ar[rr]^{f^{i,j}_{\mathcal{L}}} \ar@{^{(}->}[d]_{\iota_i} && \mathcal{L}_j \ar@{^{(}->}[d]^{\iota_j} \\
\mathcal{K}_i \ar[rr]_{f^{i,j}_{\mathcal{K}}} && \mathcal{K}_j.
}
\end{equation*}
Here, $\iota_i$ and $\iota_j$ are the canonical inclusion maps at their respective scales, while $f^{i,j}_{\mathcal{L}}$ and $f^{i,j}_{\mathcal{K}}$ are the transition maps of the persistence simplicial complexes $\mathcal{L}$ and $\mathcal{K}$. 

Consequently, the pair $(\mathcal{K}, \mathcal{L})$ defines a \textit{persistence relative complex}. For any $i \leq j$, the transition maps induce a morphism of pairs $f^{i,j}\colon (\mathcal{K}_i, \mathcal{L}_i) \to (\mathcal{K}_j, \mathcal{L}_j)$, which in turn yields a sequence of homomorphisms between the relative homology groups
\begin{equation*}
\bar{f}^{i,j}_n\colon H_n(\mathcal{K}_i, \mathcal{L}_i) \to H_n(\mathcal{K}_j, \mathcal{L}_j).
\end{equation*}
The $(i,j)$-persistent local homology is then characterized by the image of this induced map.

\begin{definition}
The \textit{$(i,j)$-persistent $n$-th relative homology group} of the pair $(\mathcal{K}, \mathcal{L})$, denoted by $H_n^{i,j}(\mathcal{K}, \mathcal{L})$, is defined as the image of this induced homomorphism
\[
H_n^{i,j}(\mathcal{K}, \mathcal{L}) := \im \left( \bar{f}^{i,j}_{n}\colon H_n(\mathcal{K}_i, \mathcal{L}_i) \to H_n(\mathcal{K}_j, \mathcal{L}_j) \right).
\]
\end{definition}
The rank of this group, $\beta_n^{i,j}(\mathcal{K}, \mathcal{L}) = \dim H_n^{i,j}(\mathcal{K}, \mathcal{L})$, is referred to as the \textit{$(i,j)$-persistent relative Betti number}. The collection of relative homology groups $\{H_n(\mathcal{K}_i, \mathcal{L}_i)\}_{i \in \mathbb{Z}}$ together with the induced maps $\{\bar{f}^{i,j}_n\}_{i \leq j}$ constitutes a persistence module. This module structure ensures that the relative topological features can be decomposed into a barcode or a persistence diagram, providing a multi-scale signature of the local environment around a site of interest.

\subsubsection{Persistent local homology}

The construction of persistent local homology relies on the dynamic tracking of a vertex and its complement across the filtration. To formalize this, we introduce the concept of a persistence vertex complex.

\begin{definition}
A \textit{persistence vertex complex} $\mathbf{v}\colon (\mathbb{Z}, \leq) \to \mathbf{Simp}$ is a persistence simplicial complex such that for each $i \in \mathbb{Z}$, the value $\mathbf{v}_i$ is either a single vertex $\{v_i\}$ or the empty set $\emptyset$. Specifically, there exists a birth index $b \in \mathbb{Z}$ such that
\begin{equation*}
\mathbf{v}_i = \begin{cases} 
\emptyset & \text{if } i < b \\
\{v_i\} & \text{if } i \geq b
\end{cases}
\end{equation*}
For all $j \geq i \geq b$, the transition morphism $f^{i,j}_{\mathbf{v}}\colon \mathbf{v}_i \to \mathbf{v}_j$ is the unique bijection between the single-vertex sets $\{v_i\}$ and $\{v_j\}$, i.e., $f^{i,j}_{\mathbf{v}}(v_i) = v_j$.
\end{definition}

Suppose $\mathbf{v}\colon (\mathbb{Z}, \leq) \to \mathbf{Simp}$ is a persistence vertex subcomplex of $\mathcal{K}\colon (\mathbb{Z}, \leq) \to \mathbf{Simp}$. We define the persistence simplicial complex $\mathcal{K}^{\mathbf{v}}\colon (\mathbb{Z}, \leq) \to \mathbf{Simp}$, where each
\begin{equation*}
  \mathcal{K}^{\mathbf{v}}_i = \mathcal{K}_i \setminus \st_{\mathcal{K}_i}(\mathbf{v}_i) = \{\sigma\in \mathcal{K}_i\mid v_i\notin \sigma\}
\end{equation*}
is the subcomplex of $\mathcal{K}_i$ consisting of all simplices that do not contain the vertex in $\mathbf{v}_i$. In particular, if $\mathbf{v}_{i}=\emptyset$, we set $\mathcal{K}^{\mathbf{v}}_i = \mathcal{K}_i$.

\begin{proposition}
Let $\mathcal{K}$ be a persistence simplicial complex and $\mathbf{v}$ a persistence vertex subcomplex of $\mathcal{K}$. Then the persistence simplicial complex $\mathcal{K}^{\mathbf{v}}\colon (\mathbb{Z}, \leq) \to \mathbf{Simp}$ is a persistence subcomplex of $\mathcal{K}$.
\end{proposition}

\begin{proof}
To establish that $\mathcal{K}^{\mathbf{v}}$ is a persistence subcomplex, we first verify that each $\mathcal{K}^{\mathbf{v}}_i$ is a well-defined subcomplex of $\mathcal{K}_i$. For any simplex $\sigma \in \mathcal{K}^{\mathbf{v}}_i$, it follows from the definition that $v_i \notin \sigma$. If $\tau$ is a face of $\sigma$, then $\tau$ is necessarily an element of $\mathcal{K}_i$. Since $\tau \subseteq \sigma$ and $v_i \notin \sigma$, the vertex $v_i$ cannot be contained in $\tau$. Thus, $\tau \in \mathcal{K}^{\mathbf{v}}_i$, confirming that $\mathcal{K}^{\mathbf{v}}_i$ is closed under the face map.

Next, we demonstrate that the transition morphisms of $\mathcal{K}$ restrict to morphisms of $\mathcal{K}^{\mathbf{v}}$. For any $j \geq i \geq b$, let $f^{i,j}\colon \mathcal{K}_i \to \mathcal{K}_j$ be the transition simplicial map. For any simplex $\sigma = [u_0, \dots, u_k] \in \mathcal{K}^{\mathbf{v}}_i$, its image under $f^{i,j}$ is the simplex in $\mathcal{K}_j$ spanned by the vertices $\{f^{i,j}(u_0), \dots, f^{i,j}(u_k)\}$. By the definition of the persistence vertex complex $\mathbf{v}$, the map $f^{i,j}$ restricts to a bijection between $\{v_i\}$ and $\{v_j\}$, implying that for any vertex $u \in \mathcal{K}_i$, $f^{i,j}(u) = v_j$ if and only if $u = v_i$. Since $v_i \notin \sigma$ by assumption, it follows that $v_j \notin f^{i,j}(\sigma)$. Consequently, $f^{i,j}$ maps $\mathcal{K}^{\mathbf{v}}_i$ into $\mathcal{K}^{\mathbf{v}}_j$. The cases where $i < b$ or $j < b$ is trivial.

Finally, since the restriction $g^{i,j} = f^{i,j}|_{\mathcal{K}^{\mathbf{v}}_i}$ is inherited from the simplicial maps, the functorial axioms of identity and composition are satisfied automatically. Specifically, for any $i \leq j \leq k$, the relation $g^{j,k} \circ g^{i,j} = (f^{j,k} \circ f^{i,j})|_{\mathcal{K}^{\mathbf{v}}_i} = f^{i,k}|_{\mathcal{K}^{\mathbf{v}}_i} = g^{i,k}$ holds, ensuring that the following diagram commutes:
\begin{equation*}
\xymatrix{
\mathcal{K}^{\mathbf{v}}_i \ar[rr]^{g^{i,j}} \ar@{^{(}->}[d]_{\iota_i} && \mathcal{K}^{\mathbf{v}}_j \ar[rr]^{g^{j,k}} \ar@{^{(}->}[d]_{\iota_j} && \mathcal{K}^{\mathbf{v}}_k \ar@{^{(}->}[d]^{\iota_k} \\
\mathcal{K}_i \ar[rr]_{f^{i,j}} && \mathcal{K}_j \ar[rr]_{f^{j,k}} && \mathcal{K}_k
}
\end{equation*}
This identifies $\mathcal{K}^{\mathbf{v}}$ as a consistent subfunctor of $\mathcal{K}$, as required.
\end{proof}

\begin{definition}
Let $\mathcal{K}$ be a persistence simplicial complex and $\mathbf{v}$ a persistence vertex subcomplex of $\mathcal{K}$. The \textit{persistent local homology} of $\mathcal{K}$ at $\mathbf{v}$ is the persistent relative homology of the pair $(\mathcal{K}, \mathcal{K}^{\mathbf{v}})$. For $i \leq j$ and $n\geq 0$, the $(i,j)$-persistent local homology group is given by
\[
H_n^{i,j}(\mathcal{K}, \mathbf{v}) := \im \left( \bar{f}^{i,j}_{n}\colon H_n(\mathcal{K}_i, \mathcal{K}_i \setminus \st_{\mathcal{K}_i}(\mathbf{v}_i)) \to H_n(\mathcal{K}_j, \mathcal{K}_j \setminus \st_{\mathcal{K}_j}(\mathbf{v}_j)) \right).
\]
\end{definition}

This definition captures how the local topological environment surrounding the evolving vertex $\mathbf{v}$ matures and stabilizes across the filtration. By mapping the local homological classes from scale $i$ to scale $j$, we can distinguish between transient local noise and robust local features that define the site's singularity or manifold-like structure.

\subsection{Spectral theory of DG-morphisms}

\subsubsection{The Moore-Penrose pseudoinverse}

\begin{definition}
A \textit{differential graded (DG) inner product space} is a tuple $(V, d, \langle \cdot, \cdot \rangle)$, where $V = \bigoplus_{n \in \mathbb{Z}} V_n$ is a graded vector space. Each component $V_n$ is equipped with an inner product $\langle \cdot, \cdot \rangle_{V_n}$ such that $V_n \perp V_m$ for $n \neq m$. The differential $d\colon V \to V$ is a linear operator of degree $1$ satisfying the condition $d^2 = 0$. 
\end{definition}

\begin{definition}
A map $f\colon V \to W$ between two DG-inner product spaces is a \textit{DG-morphism} if it is a linear map of degree zero (i.e., $f(V_n) \subseteq W_n$ for all $n$) such that $d_W f = f d_V$.
\end{definition}

\begin{definition}
Let $f\colon V \to W$ be a linear map between inner product spaces. The \textit{Moore-Penrose pseudoinverse} of $f$ is the unique linear map $f^\dagger\colon W \to V$ satisfying the following four Penrose conditions: $(i)$ $f f^\dagger f = f$; $(ii)$ $f^\dagger f f^\dagger = f^\dagger$; $(iii)$ $(f f^\dagger)^* = f f^\dagger$; $(iv)$ $(f^\dagger f)^* = f^\dagger f$.
\end{definition}

\begin{proposition}
Let $V$ and $W$ be finite-dimensional DG-inner product spaces. For any morphism $f\colon V \to W$, the Moore-Penrose pseudoinverse $f^\dagger\colon W \to V$ exists and is unique. 
\end{proposition}

\begin{proof}
Since $V$ and $W$ are finite-dimensional, $f$ can be decomposed into a direct sum of linear maps $f_n\colon V_n \to W_n$. For each $n$, the existence of $f_n^\dagger$ is guaranteed by the singular value decomposition theorem in inner product spaces. Thus, $f^\dagger = \bigoplus_{n \in \mathbb{Z}} f_n^\dagger$ provides the unique pseudoinverse for the entire graded map.
\end{proof}

\begin{proposition}\label{proposition:properties}
Let $f\colon V \to W$ be a DG-morphism between finite-dimensional DG-inner product spaces. The Moore-Penrose pseudoinverse $f^\dagger\colon W \to V$ satisfies the following fundamental properties:
\begin{enumerate}[label=$(\roman*)$]
    \item The compositions $P_{\im (f)} = f f^\dagger$ and $P_{(\ker f)^\perp} = f^\dagger f$ are orthogonal projections onto the range of $f$ in $W$ and the orthogonal complement of the kernel in $V$, respectively.
    \item For any $w \in W$, the vector $v = f^\dagger w$ is the unique minimum-norm least-squares solution to the equation $f(v) = w$. That is, $f^\dagger w = \arg\min \{ \|u\| \colon u \in \arg\min_{z \in V} \|f(z) - w\| \}$.
    \item The pseudoinversion operator is compatible with the adjoint, such that $(f^*)^\dagger = (f^\dagger)^*$.
    \item If $f$ is injective, then $f^\dagger = (f^* f)^{-1} f^*$. In the particular case where $f$ is an isometric embedding, the pseudoinverse reduces to the adjoint, $f^\dagger = f^*$.
\end{enumerate}
\end{proposition}

\begin{proof}
The relevant properties of the Moore-Penrose pseudoinverse are classical \cite{barata2012moore,ben2003generalized}. For completeness and to ensure clarity in this paper, we include the proofs in our setting.

$(i)$ Let $P = f f^\dagger$. From the first and third Penrose conditions, we have $P^2 = (f f^\dagger f) f^\dagger = f f^\dagger = P$ and $P^* = (f f^\dagger)^* = f f^\dagger = P$. Thus, $P$ is an orthogonal projection. Since $\im (f f^\dagger) \subseteq \im (f)$ and $f = f f^\dagger f$ implies $\im (f) \subseteq \im (f f^\dagger)$, we conclude $\im (P) = \im (f)$. A symmetric argument shows that $f^\dagger f$ is the orthogonal projection onto $\im (f^\dagger)$. By the fourth condition, $\im (f^\dagger) = \im (f^\dagger f) = \ker(f^\dagger f)^\perp = (\ker f)^\perp$.

$(ii)$ For any $w \in W$, let $v = f^\dagger w$. Any solution $u$ to the least-squares problem must satisfy the normal equation $f^* f u = f^* w$. Substituting $w = f v + (w - f v)$, where $w - f v = (I - f f^\dagger) w \in \im (f)^\perp = \ker(f^*)$, we see that $v$ satisfies the normal equation. Since $\im (f^\dagger) = (\ker f)^\perp$, $v$ is orthogonal to $\ker(f)$. By the Pythagorean theorem, for any other least-squares solution $v' = v + z$ with $z \in \ker(f)$, we have $\|v'\|^2 = \|v\|^2 + \|z\|^2 > \|v\|^2$ unless $z=0$.

$(iii)$ Let $G = (f^\dagger)^*$. We verify that $G$ satisfies the four Penrose conditions for $f^*$ as follows:
\begin{align*}
    f^* G f^* &= f^* (f^\dagger)^* f^* = (f f^\dagger f)^* = f^*, \\
    G f^* G &= (f^\dagger)^* f^* (f^\dagger)^* = (f^\dagger f f^\dagger)^* = (f^\dagger)^*, \\
    (f^* G)^* &= (f^* (f^\dagger)^*)^* = f^\dagger f, \\
    (G f^*)^* &= ((f^\dagger)^* f^*)^* = f f^\dagger.
\end{align*}
Since $f^\dagger f$ and $f f^\dagger$ are self-adjoint, the conditions $(f^* G)^* = f^* G$ and $(G f^*)^* = G f^*$ hold. By the uniqueness of the pseudoinverse, $(f^*)^\dagger = (f^\dagger)^*$.

$(iv)$ If $f$ is injective, then $f^* f$ is an invertible operator on $V$. It is straightforward to verify that $g = (f^* f)^{-1} f^*$ satisfies the four conditions. For instance, $f g f = f (f^* f)^{-1} (f^* f) = f$. For the isometric case, $f^* f = \text{id}_V$, which directly implies $f^\dagger = (\text{id}_V)^{-1} f^* = f^*$.
\end{proof}

\subsubsection{The generalized morphism Laplacian}

From now on, for convenience, the DG-inner product spaces considered are always assumed to be finite-dimensional. 

Let $(V,d_V,\langle \cdot, \cdot \rangle)$ and $(W,d_W,\langle \cdot, \cdot \rangle)$ be DG-inner product spaces, and let $f\colon V\to W$ be a DG-morphism. The \textit{persistence domain} of $f$ is the space 
\begin{equation*}
\Theta_{f} := \{x \in W \mid d_W x \in \im f \}.
\end{equation*}
Let $\iota\colon \Theta_f \hookrightarrow W$ denote the canonical inclusion.
\begin{equation*}
\xymatrix@1{ \Theta_{f} \ar[r]^{\iota} & W \ar[r]^{d_W} & W \ar[r]^{f^\dagger} & V }.
\end{equation*}
\begin{definition}
Let $f\colon V\to W$ be a DG-morphism between $V$ and $W$.
The \textit{generalized pullback differential} $\delta_f\colon \Theta_f \to V$ as
\begin{equation*}
\delta_f := f^\dagger d_W \iota.
\end{equation*}
\end{definition}

This construction yields the following sequence
\begin{equation*}
\xymatrix@1{ \Theta_{f} \ar[r]^{\delta_f} & V \ar[r]^{d_V} & V \ar[r]^{f} & W }.
\end{equation*}

\begin{lemma}
For any $x \in \Theta_f$, let $\delta_f(x) = f^\dagger d_W \iota(x)$. Then we have $d_V \delta_f(x) \in \ker(f)$. In particular, if $f$ is injective, then $d_V \delta_f = 0$.
\end{lemma}

\begin{proof}
By the definition of the persistence domain $\Theta_f$, for any $x \in \Theta_f$, there exists some $v \in V$ such that $d_W \iota(x) = f(v)$. Applying the Moore-Penrose pseudoinverse $f^\dagger$ to both sides, we obtain
\begin{equation*}
    \delta_f(x) = f^\dagger f(v) = P_{(\ker f)^\perp}(v).
\end{equation*}
To verify the consistency with the differential $d_V$, we examine the image of $\delta_f(x)$ under $f \circ d_V$ as follows:
\begin{align}
    f d_V \delta_f(x) &= f d_V P_{(\ker f)^\perp}(v) \nonumber \\
    &= f d_V v \label{eq:proj_nroperty} \\
    &= d_W f(v) \label{eq:morphism_nroperty} \\
    &= d_W^2 \iota(x) = 0. \nonumber
\end{align}
In the above derivation, \eqref{eq:proj_nroperty} follows from the fact that $v - P_{(\ker f)^\perp}(v) \in \ker(f)$ and the morphism property $d_W f = f d_V$ implies $f d_V (\ker f) \subseteq d_W f (\ker f) = 0$. Step \eqref{eq:morphism_nroperty} further utilizes the intertwining property of the morphism $f$. The identity $f(d_V \delta_f(x)) = 0$ implies that $d_V \delta_f(x) \in \ker(f)$. If $f$ is injective, $\ker(f) = \{0\}$, which directly yields $d_V \delta_f = 0$.
\end{proof}

\begin{remark}
While $f$ satisfies the morphism condition $d_W f = f d_V$, its pseudoinverse $f^\dagger$ generally does not satisfy $d_V f^\dagger = f^\dagger d_W$. However, the composite operator $\delta_f = f^\dagger d_W \iota$ remains a valid pullback differential for the construction of morphism Laplacians provided that $d_V \delta_f = 0$ is satisfied on the relevant subspace $(\ker f)^\perp$.
\end{remark}

\begin{definition}
The \textit{generalized Laplacian} $\Delta^f\colon V \to V$ associated with $f\colon V \to W$ is the operator
\begin{equation*}
\Delta^f := \delta_f \delta_f^* +  d_V^* d_V + (\mathrm{id} - f^\dagger f) .
\end{equation*}
\end{definition}

\begin{proposition}
The generalized Laplacian $\Delta^f\colon V \to V$ is a self-adjoint and positive semi-definite operator.
\end{proposition}

\begin{proof}
The self-adjointness of $\Delta^f$ follows from the self-adjointness of the three operators $\delta_f \delta_f^*$, $d_V^* d_V$, and $\mathrm{id} - f^\dagger f$. Here, $\mathrm{id} - f^\dagger f$ is self-adjoint since the Moore-Penrose pseudoinverse satisfies $(f^\dagger f)^* = f^\dagger f$.

Next, we demonstrate the positive semi-definiteness. For any vector $v \in V$, the quadratic form associated with $\Delta^f$ is given by
\begin{align*}
    \langle \Delta^f v, v \rangle_V &= \langle \delta_f \delta_f^* v, v \rangle_V + \langle d_V^* d_V v, v \rangle_V + \langle (\mathrm{id} - f^\dagger f) v, v \rangle_V \\
    &= \langle \delta_f^* v, \delta_f^* v \rangle_V + \langle d_V v, d_V v \rangle_V + \langle (\mathrm{id} - f^\dagger f)^2 v, v \rangle_V \\
    &= \|\delta_f^* v\|_V^2 + \|d_V v\|_V^2 + \|(\mathrm{id} - f^\dagger f) v\|_V^2,
\end{align*}
where we utilized the fact that $\mathrm{id} - f^\dagger f$ is an idempotent self-adjoint operator. Since the squared norm of any vector in an inner product space is non-negative, it follows that $\langle \Delta^f v, v \rangle_V \geq 0$ for all $v \in V$. This completes the proof.
\end{proof}

\subsubsection{Generalized persistent harmonic space}

Let $V$ and $W$ be DG-inner product spaces, and let $f\colon V \to W$ be a DG-morphism.

\begin{definition}
The \textit{persistent harmonic space} associated with the DG-morphism $f\colon V \to W$, denoted by $\mathcal{H}(f)$, is defined as the kernel of the generalized Laplacian operator
\begin{equation*}
    \mathcal{H}^{f}:= \ker(\Delta^f) = \{ v \in V \mid \Delta^f v = 0 \}.
\end{equation*}
Elements of $\mathcal{H}(f)$ are referred to as \textit{$f$-harmonic representatives}.
\end{definition}

\begin{proposition}\label{proposition:generalized_harmonic}
The persistent  harmonic space $\mathcal{H}^{f}\coloneqq \ker \Delta^f$ is characterized by
\begin{equation*}
    \mathcal{H}^{f}=  \ker d_V \cap (\im \delta_f)^\perp \cap (\ker f)^\perp.
\end{equation*}
\end{proposition}

\begin{proof}
By the self-adjoint and positive semi-definite nature of $\Delta^f$, an element $v \in V$ belongs to the kernel $\mathcal{H}(f)$ if and only if the associated quadratic form vanishes
\begin{equation*}
    \langle \Delta^f v, v \rangle = \|\delta_f^* v\|^2 + \|d_V v\|^2 + \|(\mathrm{id} - f^\dagger f)v\|^2 = 0.
\end{equation*}
This identity holds if and only if each non-negative term vanishes simultaneously:
\begin{enumerate}[label=$(\roman*)$]
    \item $\|d_V v\|^2 = 0 \iff v \in \ker d_V$, ensuring $v$ is a cycle in $V$;
    \item $\|\delta_f^* v\|^2 = 0 \iff v \in (\im \delta_f)^\perp$ by the fundamental relation $\ker \delta_f^* = (\im \delta_f)^\perp$;
    \item $\|(\mathrm{id} - f^\dagger f)v\|^2 = 0 \iff v \in \ker(\mathrm{id} - f^\dagger f)$.
\end{enumerate}
Recall that $\mathrm{id} - f^\dagger f$ is the orthogonal projection onto $\ker f$; thus, $\ker(\mathrm{id} - f^\dagger f) = (\ker f)^\perp$. Consequently, we obtain the desired result $\mathcal{H}^{f}= \ker d_V \cap (\im \delta_f)^\perp \cap (\ker f)^\perp$.
\end{proof}

In the particular case where $f\colon V \to W$ is an isometric embedding, the Moore-Penrose pseudoinverse satisfies $f^\dagger f = \mathrm{id}_V$. Consequently, the generalized Laplacian $\Delta^f$ reduces to
\begin{equation*}
    \Delta^f = \delta_f \delta_f^* + d_V^* d_V,
\end{equation*}
which is precisely the definition of the persistent Laplacian introduced in \cite{memoli2022persistent, wang2020persistent}. Furthermore, under this injectivity assumption, the kernel $\ker \Delta^f$ coincides with the persistent harmonic space characterized in \cite{liu2024algebraic}. Thus, our construction can be viewed as a direct generalization of these constructions to morphisms with non-trivial kernels.

\subsubsection{Persistent Hodge theorem}

Let $f \colon V \to W$ be a DG-morphism between DG-inner product spaces $V$ and $W$. The morphism $f$ naturally induces a linear map between the corresponding homology groups:
\begin{equation*}
    H(f) \colon H(V) \to H(W), \quad [v] \mapsto [f(v)].
\end{equation*}
The \textit{persistent homology} associated with $f$ is defined as the image of this induced map
\begin{equation*}
    H^{f} = \im \left( H(f) \colon H(V) \to H(W) \right) \subseteq H(W).
\end{equation*}
In the specific case where $f$ is an isometric embedding, there exists a natural isomorphism between the persistent harmonic space and the persistent homology. However, the situation becomes considerably more involved when $f$ is an arbitrary DG-morphism. To address these complexities, we provide an explicit characterization of this relationship in the following discussion.

\begin{theorem}[Hodge isomorphism theorem]\label{theorem:hodge_isomorphism}
The linear map
\begin{equation*}
  \Psi\colon \ker \Delta^f \to H^{f}, \quad  v\mapsto [f(v)]
\end{equation*}
is an isomorphism of linear spaces.
\end{theorem}

\begin{proof}
By Proposition \ref{proposition:generalized_harmonic}, the kernel of $\Delta^f$ coincides with the intersection of the kernels of its constituent summands
\begin{equation*}
\ker \Delta^f = \ker(\delta_f^*) \cap \ker(d_V) \cap \ker(\mathrm{id} - f^\dagger f).
\end{equation*}
Accordingly, an element $v$ belongs to $\ker \Delta^f$ if and only if it satisfies the following conditions:
\begin{itemize}[label=$(\roman*)$]
    \item[(i)] $v \in Z(V) = \ker d_V$, identifying $v$ as a cycle in $V$;
    \item[(ii)] $v \in (\ker f)^\perp$, which follows from $(\mathrm{id} - f^\dagger f)v = 0$;
    \item[(iii)] $v \in \ker \delta_f^* = (\im \delta_f)^\perp$, i.e., $\langle v, f^\dagger d_W z \rangle_V = 0$ for all $z \in \Theta_f$.
\end{itemize}

\textit{Injectivity:} 
Suppose $v \in \ker \Delta^f$ such that $\Psi(v) = [f(v)] = 0$ in $H(W)$. Then $f(v)$ is a boundary in $W$, implying there exists some $x \in \Theta_f$ such that $f(v) = d_W x$. Since $f f^\dagger$ is the orthogonal projection onto $\im f$ and $d_W x \in \im f$, we have $f f^\dagger d_W x = d_W x$. Utilizing condition (iii), we observe
\begin{equation*}
\| f(v) \|_W^2 = \langle f(v), d_W x \rangle_W = \langle f(v), f f^\dagger d_W x \rangle_W = \langle f^* f v, f^\dagger d_W x \rangle_V = \langle v, f^\dagger d_W x \rangle_V = 0.
\end{equation*}
This necessitates $f(v) = 0$, whence $v \in \ker f$. Combined with condition (ii), which requires $v \in (\ker f)^\perp$, we conclude $v \in \ker f \cap (\ker f)^\perp = \{0\}$. Thus, $\Psi$ is injective.

\textit{Surjectivity:} 
Let $\xi \in H^{f} \subseteq H(W)$. By definition, there exists a cycle $c \in Z(V)$ such that $[f(c)] = \xi$. Let $c' = f^\dagger f (c)$ be the orthogonal projection of $c$ onto $(\ker f)^\perp$. Since $f(c') = f(c)$, it follows that $[f(c')] = \xi$. We now apply the orthogonal decomposition of $Z(V)$ associated with the operator $\delta_f \delta_f^*$ as follows
\begin{equation*}
Z(V) = \im(\delta_f|_{Z(V)}) \oplus (\ker \delta_f^* \cap Z(V)).
\end{equation*}
Decomposing $c' = v + \delta_f z$ where $v \in \ker \delta_f^* \cap Z(V)$ and $z \in \Theta_f$, we observe that both $c'$ and $\delta_f z = f^\dagger d_W z$ reside in $(\ker f)^\perp$. Consequently, their difference $v$ must also lie in $(\ker f)^\perp$. Hence, $v$ satisfies conditions (i), (ii), and (iii), proving $v \in \ker \Delta^f$. Finally, we have
\begin{equation*}
f(v) = f(c' - \delta_f z) = f(c') - f f^\dagger d_W z = f(c') - d_W z.
\end{equation*}
In homology, $[f(v)] = [f(c')] - [d_W z] = [f(c')] = \xi$, which establishes the surjectivity of $\Psi$.
\end{proof}

When $f$ is an isometric embedding, $\Delta^f$ recovers the classical persistent Laplacian, and the theorem reduces to the standard persistent Hodge isomorphism.

\begin{remark}
A similar extension of the persistent Laplacian to general simplicial maps was explored in \cite{gulen2023generalization}. The core strategy therein involves restricting the morphism $f \colon V \to W$ to a smaller subspace, thereby inducing an isometric isomorphism $\hat{f} \colon (\ker f)^{\perp} \to \im f$ to facilitate the definition of the persistent Laplacian. In contrast, our construction does not impose such isometric constraints on the morphism $f$ itself. This approach is more consistent with the intrinsic nature of persistent homology, which is defined independently of any specific inner product structure. 
\end{remark}

\begin{example}
Consider the $1$-skeleton of a simplicial complex $K$ consisting of two triangles $\{v_0, v_1, v_2\}$ and $\{v_1, v_2, v_3\}$ sharing a common edge $e_{12}$. The target space $L$ is a single triangle $\{w_0, w_1, w_2\}$. The $1$-cycle space $Z_1(K)$ is spanned by the two fundamental cycles
\begin{equation*}
    \gamma_1 = e_{01} + e_{12} - e_{02}, \quad \gamma_2 = e_{12} + e_{23} - e_{13},
\end{equation*}
where $e_{ij}$ denotes the oriented edge from $v_i$ to $v_j$. We define a simplicial map $\phi\colon K \to L$ via the vertex assignments $\phi(v_0) = w_0$, $\phi(v_1) = w_1$, $\phi(v_2) = w_2$, and $\phi(v_3) = w_0$.
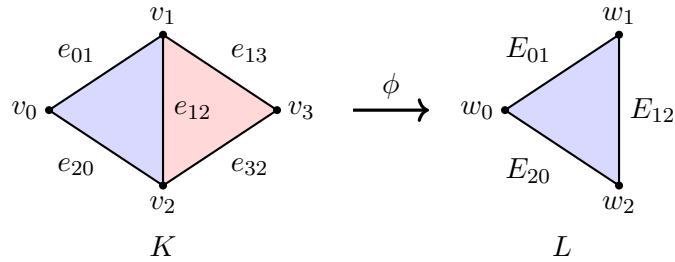
\begin{figure}[ht]
    \centering
    \begin{tikzpicture}[scale=1, thick] 
        \begin{scope}
            \coordinate (v1) at (0,1);
            \coordinate (v2) at (0,-1);
            \coordinate (v0) at (-1.5,0);
            \coordinate (v3) at (1.5,0);
            
            \fill[blue!15] (v0) -- (v1) -- (v2) -- cycle;
            \fill[red!15] (v3) -- (v1) -- (v2) -- cycle;
            
            \draw (v0) -- (v1) node[midway, left, anchor=south east] {$e_{01}$};
            \draw (v1) -- (v2) node[midway, right] {$e_{12}$};
            \draw (v2) -- (v0) node[midway, left, anchor=north east] {$e_{20}$};
            \draw (v1) -- (v3) node[midway, right, anchor=south west] {$e_{13}$};
            \draw (v3) -- (v2) node[midway, right, anchor=north west] {$e_{32}$};
            
            \fill (v0) circle (1.5pt) node[left] {$v_0$};
            \fill (v1) circle (1.5pt) node[above] {$v_1$};
            \fill (v2) circle (1.5pt) node[below] {$v_2$};
            \fill (v3) circle (1.5pt) node[right] {$v_3$};
            \node at (0,-1.8) {$K$};
        \end{scope}

        \draw[->, very thick] (2.5,0) -- (3.5,0) node[midway, above] {$\phi$};

        \begin{scope}[xshift=6cm]
            \coordinate (w1) at (0,1);
            \coordinate (w2) at (0,-1);
            \coordinate (w0) at (-1.5,0);
            
            \fill[blue!15] (w0) -- (w1) -- (w2) -- cycle;
            
            \draw (w0) -- (w1) node[midway, left, anchor=south east] {$E_{01}$};
            \draw (w1) -- (w2) node[midway, right] {$E_{12}$};
            \draw (w2) -- (w0) node[midway, left, anchor=north east] {$E_{20}$};
            
            \fill (w0) circle (1.5pt) node[left] {$w_0$};
            \fill (w1) circle (1.5pt) node[above] {$w_1$};
            \fill (w2) circle (1.5pt) node[below] {$w_2$};
            \node at (-0.75,-1.8) {$L$};
        \end{scope}
    \end{tikzpicture}
    \caption{Illustration of the simplicial map $\phi\colon K \to L$.}
\end{figure}
The induced chain map $f \colon C_1(K) \to C_1(L)$ is given by
\begin{equation*}
    f(e_{01}) = E_{01}, \quad f(e_{12}) = E_{12}, \quad f(e_{02}) = E_{02}, \quad f(e_{13}) = -E_{01}, \quad f(e_{23}) = -E_{02},
\end{equation*}
where $E_{ij}$ are the edges of $L$. Under the induced homology map $H_1(\phi)$, the images of the cycles are given by $H_1(\phi)([\gamma_1]) = [E_{01} + E_{12} - E_{02}] = [\Gamma_L]$ and $H_1(\phi)([\gamma_2]) = [E_{12} - E_{02} + E_{01}]$. Since both cycles map to the same non-trivial class in $H_1(L)$, the persistent homology $H^f$ is a one-dimensional subspace spanned by $[\Gamma_L]$.

To verify the Hodge isomorphism, we identify the harmonic representative $v \in \ker \Delta^f$. According to Theorem \ref{theorem:hodge_isomorphism}, $v$ must reside in $Z_1(K) \cap (\ker f)^\perp$. The kernel of the chain map $f$ is spanned by the elements
\begin{equation*}
    \eta_1 = e_{01} + e_{13}, \quad \eta_2 = e_{02} + e_{23}.
\end{equation*}
Let $v = c_1 \gamma_1 + c_2 \gamma_2$ be a general element in $Z_1(K)$. The orthogonality condition $v \in (\ker f)^\perp$ requires $\langle v, \eta_1 \rangle = 0$ and $\langle v, \eta_2 \rangle = 0$. Under the standard inner product where the edges form an orthonormal basis, these conditions yield
\begin{equation*}
    \langle c_1 \gamma_1 + c_2 \gamma_2, e_{01} + e_{13} \rangle = c_1 - c_2 = 0, \quad \langle c_1 \gamma_1 + c_2 \gamma_2, e_{02} + e_{23} \rangle = - c_1 + c_2 = 0.
\end{equation*}
Thus, we must have $c_1 = c_2$, which implies that the harmonic space $\ker \Delta^f$ is one-dimensional and spanned by $\gamma_1 + \gamma_2$. This result implies that $\dim(\ker \Delta^f) = \dim(H^f) = 1$. More precisely, the linear map
\begin{equation*}
   \Psi \colon \mathrm{span}\{\gamma_1 + \gamma_2\} \to \mathrm{span}\{[\Gamma_{L}]\}, \quad v \mapsto [f(v)]
\end{equation*}
is an isomorphism between the harmonic space and the persistent homology. 
\end{example}

\subsection{Persistent local Laplacian}

\subsubsection{Persistent relative Laplacian}

Let $\mathcal{K} \colon (\mathbb{Z}, \leq) \to \mathbf{Simp}$ be a persistence simplicial complex, and let $\mathcal{L} \colon (\mathbb{Z}, \leq) \to \mathbf{Simp}$ be a persistence subcomplex of $\mathcal{K}$. For any $i \leq j$, the map of simplicial pairs $f^{i,j}\colon (\mathcal{K}_i, \mathcal{L}_i) \hookrightarrow (\mathcal{K}_j, \mathcal{L}_j)$ induces a natural morphism of relative chain complexes
\begin{equation*}
    \bar{f}^{i,j} \colon C_{\ast}(\mathcal{K}_i, \mathcal{L}_i) \to C_{\ast}(\mathcal{K}_j, \mathcal{L}_j).
\end{equation*}
By endowing these complexes with the quotient inner product introduced in Section \ref{section:local_Laplacian}, $\bar{f}^{i,j}$ becomes a DG-morphism of differential graded inner product spaces.

To define the persistent Laplacian in this relative setting, we follow the generalized construction by considering the persistence domain of $\bar{f}^{i,j}$ at degree $n+1$, defined as
\begin{equation*}
    \Theta^{i,j}_{n+1} := \{x \in C_{n+1}(\mathcal{K}_j, \mathcal{L}_j) \mid d^j_{n+1} x \in \mathrm{im}\, \bar{f}^{i,j} \}.
\end{equation*}
Here, $d^i_\ast$ denotes the relative boundary operator of $C_\ast(\mathcal{K}_i, \mathcal{L}_i)$, while $d^j_\ast$ refers to the corresponding operator for $C_\ast(\mathcal{K}_j, \mathcal{L}_j)$.
Let $\iota \colon \Theta^{i,j}_{n+1} \hookrightarrow C_{n+1}(\mathcal{K}_j, \mathcal{L}_j)$ be the canonical inclusion. The corresponding generalized pullback differential $\delta^{i,j}_{n+1} \colon \Theta^{i,j}_{n+1} \to C_n(\mathcal{K}_i, \mathcal{L}_i)$ is given by
\begin{equation*}
    \delta^{i,j}_{n+1} := (\bar{f}^{i,j}_{n})^\dagger d^j_{n+1} \iota.
\end{equation*}
The relative persistent Laplacian is then established as the generalized Laplacian associated with the morphism $\bar{f}^{i,j}$ in the following sense.

\begin{definition}
The \textit{$(i,j)$-persistent relative Laplacian} in dimension $n$, denoted by $\Delta^{i,j}_{n} \colon C_{n}(\mathcal{K}_i, \mathcal{L}_i) \to C_{n}(\mathcal{K}_i, \mathcal{L}_i)$, is the linear operator defined by
\begin{equation*}
    \Delta^{i,j}_{n} := \delta^{i,j}_{n+1} (\delta^{i,j}_{n+1})^* + (d^i_n)^* d^i_n + (\mathrm{id} - (\bar{f}^{i,j}_{n})^\dagger \bar{f}^{i,j}_{n}),
\end{equation*}
where $d^i_n$ is the relative boundary operator of $C_{\ast}(\mathcal{K}_i, \mathcal{L}_i)$, and $(\delta^{i,j}_{n+1})^*$ is the adjoint of the generalized pullback differential.
\end{definition}
In particular, when $n=0$, the relative boundary operator $d^i_0$ vanishes by definition, and the $(i,j)$-persistent relative Laplacian simplifies to
\begin{equation*}
    \Delta_0^{i,j} = \delta^{i,j}_{1} (\delta^{i,j}_{1})^* + (\mathrm{id} - (\bar{f}^{i,j}_{0})^\dagger \bar{f}^{i,j}_{0}).
\end{equation*}
Furthermore, in the case where $i=j$, the induced morphism $\bar{f}^{i,j}$ reduces to the identity. Consequently, the generalized pullback differential $\delta^{i,j}_{n+1}$ coincides with the standard relative boundary operator $d^i_{n+1}$. Under these conditions, the persistent operator recovers the combinatorial Laplacian for relative chain complexes
\begin{equation*}
    \Delta^{i,i}_{n} = d^{i}_{n+1} (d^{i}_{n+1})^* + (d^i_n)^* d^i_n,
\end{equation*}
which is exactly the standard (non-persistent) relative Laplacian $\Delta^{i}_{n} \colon C_{n}(\mathcal{K}_i, \mathcal{L}_i) \to C_{n}(\mathcal{K}_i, \mathcal{L}_i)$. 

\begin{definition}
The \textit{$(i,j)$-persistent relative harmonic space} is defined by
\begin{equation*}
  \mathcal{H}^{i,j}_n(\mathcal{K},\mathcal{L}) = \ker \Delta^{i,j}_{n},\quad n\geq 0 .
\end{equation*}
\end{definition}

\begin{theorem}\label{theorem:relative_isomorphism}
For any $i\leq j$, there is an isomorphism of linear spaces
\begin{equation*}
  \mathcal{H}^{i,j}_n(\mathcal{K},\mathcal{L}) = H_n^{i,j}(\mathcal{K}, \mathcal{L}),\quad n\geq 0.
\end{equation*}
\end{theorem}
\begin{proof}
It is directly obtained from Theorem \ref{theorem:hodge_isomorphism}.
\end{proof}

The above result demonstrates that the harmonic space of the persistent Laplacian coincides with persistent homology, thereby establishing a connection between the geometry and topology of a space or dataset. This also provides an alternative approach for calculating persistent relative homology.

\subsubsection{Persistent local Laplacian}

To facilitate the algebraic treatment of persistent Laplacians, we restrict our study to filtrations of embedded simplicial complexes. We now consider the category $\mathbf{Simp}^{\hookrightarrow}$, whose morphisms are defined by inclusions. Such a restriction simplifies the induced chain maps to canonical monomorphisms, thereby providing a robust foundation for the persistent local Laplacian while covering the diverse scenarios encountered in practice.

Let $\mathcal{K} \colon (\mathbb{Z}, \leq) \to \mathbf{Simp}^{\hookrightarrow}$ be a persistence simplicial complex. For any vertex $v \in \bigcup_{i \in \mathbb{Z}} \mathcal{K}_i$, we construct an associated \textit{persistence vertex complex} $\mathbf{v} \colon (\mathbb{Z}, \leq) \to \mathbf{Simp}$, where its value at index $i$ is defined by
\begin{equation*}
\mathbf{v}_i = \begin{cases} 
\{v\} & \text{if } i \geq b, \\
\emptyset & \text{if } i < b.
\end{cases}
\end{equation*}
Here, $b = \min\{i \in \mathbb{Z} \mid v \in \mathcal{K}_i\}$ denotes the birth index of $v$ in the filtration $\mathcal{K}$. For all $j \geq i \geq b$, the transition morphism $f^{i,j}_{\mathbf{v}} \colon \mathbf{v}_i \to \mathbf{v}_j$ is the unique identity map, i.e., $f^{i,j}_{\mathbf{v}}(v) = v$.

Let $\mathcal{K}^{\mathbf{v}}_i = \mathcal{K}_i \setminus \st_{\mathcal{K}_i}(\mathbf{v}_i) = \{\sigma\in \mathcal{K}_i\mid v_i\notin \sigma\}$. We consider the relative chain complex $C_{\ast}(\mathcal{K}_i, \mathcal{K}^{\mathbf{v}}_i)$, which, by our previous construction, is equipped with a quotient inner product. For any $i \leq j$, the inclusion of simplicial pairs $(\mathcal{K}_i, \mathcal{K}^{\mathbf{v}}_i) \hookrightarrow (\mathcal{K}_j, \mathcal{K}^{\mathbf{v}}_j)$ induces a natural chain map
\begin{equation*}
    \bar{f}^{i,j} \colon C_{\ast}(\mathcal{K}_i, \mathcal{K}^{\mathbf{v}}_i) \to C_{\ast}(\mathcal{K}_j, \mathcal{K}^{\mathbf{v}}_j).
\end{equation*}
As established previously, $f^{i,j}$ is a DG-morphism between differential graded inner product spaces.

\begin{lemma}\label{lemma:isometric_embedding}
The induced chain map $\bar{f}^{i,j} \colon C_{\ast}(\mathcal{K}_i, \mathcal{K}^{\mathbf{v}}_i) \to C_{\ast}(\mathcal{K}_j, \mathcal{K}^{\mathbf{v}}_j)$ is an isometric embedding.
\end{lemma}

\begin{proof}
    Consider the basis of $C_{\ast}(\mathcal{K}_i, \mathcal{K}^{\mathbf{v}}_i)$ consisting of simplices in $\mathcal{K}_i \setminus \mathcal{K}^{\mathbf{v}}_i$. By construction, these are precisely the simplices $\sigma \in \mathcal{K}_i$ that contain the vertex $v$. Since the transition morphism in $\mathbf{Simp}^{\hookrightarrow}$ is an inclusion, any such simplex $\sigma$ is mapped to itself in $\mathcal{K}_j$, and it still contains $v$. 

    Furthermore, the inclusion of simplicial pairs $(\mathcal{K}_i, \mathcal{K}^{\mathbf{v}}_i) \hookrightarrow (\mathcal{K}_j, \mathcal{K}^{\mathbf{v}}_j)$ ensures that the set of simplices generating $C_{\ast}(\mathcal{K}_i, \mathcal{K}^{\mathbf{v}}_i)$ is a subset of those generating $C_{\ast}(\mathcal{K}_j, \mathcal{K}^{\mathbf{v}}_j)$. Since the inner product on these relative chain complexes is defined such that the simplices form an orthonormal basis, the induced map $\bar{f}^{i,j}$ maps a subset of an orthonormal basis in the codomain to itself. Consequently, $\bar{f}^{i,j}$ preserves the inner product on the entire space by linearity. It follows that $\bar{f}^{i,j}$ is an isometric embedding.
\end{proof}

Lemma \ref{lemma:isometric_embedding} implies that $(\bar{f}^{i,j})^\dagger \bar{f}^{i,j} = \mathrm{id}$, which ensures that the term $\mathrm{id} - (\bar{f}^{i,j})^\dagger \bar{f}^{i,j}$ in the generalized persistent Laplacian vanishes. Consequently, we obtain the following simplified definition for the local setting.

\begin{definition}
For any $i \leq j$, the \textit{$(i,j)$-persistent local Laplacian} in dimension $n$, denoted by $\Delta_n^{i,j} \colon C_n(\mathcal{K}_i, \mathcal{K}^{\mathbf{v}}_i) \to C_n(\mathcal{K}_i, \mathcal{K}^{\mathbf{v}}_i)$, is defined as
\begin{equation*}
    \Delta^{i,j}_{n} := \delta^{i,j}_{n+1} (\delta^{i,j}_{n+1})^* + (d^i_n)^* d^i_n,
\end{equation*}
where $d^i_n$ is the boundary operator of $C_{\ast}(\mathcal{K}_i, \mathcal{K}^{\mathbf{v}}_i)$, and $(\delta^{i,j}_{n+1})^*$ is the adjoint of the generalized pullback differential.
\end{definition}

\begin{definition}
For any $i \leq j$, the \textit{$(i,j)$-persistent local harmonic space} is defined by
\begin{equation*}
  \mathcal{H}^{i,j}_n(\mathcal{K},\mathbf{v}) = \ker \Delta^{i,j}_{n},\quad n\geq 0.
\end{equation*}
\end{definition}

The persistent local Laplacian extends the static local Laplacian into a multi-scale setting, enabling the spectral analysis of localized topological features across the filtration.

\begin{theorem}
For any $i\leq j$, there is an isomorphism of linear spaces
\begin{equation*}
  \mathcal{H}^{i,j}_n(\mathcal{K},\mathbf{v}) = H_n^{i,j}(\mathcal{K}, \mathbf{v}),\quad n\geq 0.
\end{equation*}
\end{theorem}
\begin{proof}
This is a direct consequence of Theorem \ref{theorem:relative_isomorphism}.
\end{proof}

Following the definition of the persistent local Laplacian, we introduce its spectral properties, which characterize the evolution of local topological features.

\begin{definition}
For any $i \leq j$, the \textit{persistent local spectrum} in dimension $n$, denoted by $\mathbf{Spec}(\Delta_n^{i,j})$, is the multiset of eigenvalues of the operator $\Delta_n^{i,j}$.
\end{definition}

This spectral formulation offers a significant advantage over purely homological persistence: the non-zero eigenvalues of $\Delta_n^{i,j}$ provide a ``spectral signature'' of the local geometry. As $j-i$ increases, the evolution of these eigenvalues tracks the structural stability of the neighborhood around $v$, allowing for the differentiation between transient geometric noise and robust local singularities.

\subsubsection{Computing persistent local Laplacian using link complex}

\begin{definition}
Let $\mathcal{K} \colon (\mathbb{Z}, \leq) \to \mathbf{Simp}^{\hookrightarrow}$ be a persistence simplicial complex. For any vertex $v \in \bigcup_{i \in \mathbb{Z}} \mathcal{K}_i$, let $b = \min\{i \in \mathbb{Z} \mid v \in \mathcal{K}_i\}$. The \textit{persistent link complex} of $v$ in $\mathcal{K}$, denoted by $\mathrm{Lk}_{\mathcal{K}}(v) \colon (\mathbb{Z}, \leq) \to \mathbf{Simp}^{\hookrightarrow}$, is defined as follows:
\begin{enumerate}[label=$(\roman*)$]
    \item For $i \in \mathbb{Z}$, the link complex at index $i$ is
    \begin{equation*}
    \mathrm{Lk}_{\mathcal{K}, i}(v) := 
    \begin{cases} 
    \mathrm{Lk}_{\mathcal{K}_i}(v) & \text{if } i \geq b, \\
    \emptyset & \text{if } i < b.
    \end{cases}
    \end{equation*}
    \item For $i \leq j$, the transition morphism $g^{i,j} \colon \mathrm{Lk}_{\mathcal{K}, i}(v) \to \mathrm{Lk}_{\mathcal{K}, j}(v)$ is the restriction of the inclusion $\mathcal{K}_i \hookrightarrow \mathcal{K}_j$ to the corresponding link.
\end{enumerate}
\end{definition}

By virtue of the chain isometry $\phi$ established in Theorem \ref{theorem:chain_isomorphism}, the study of the persistent local Laplacian $\Delta_n^{i,j}$ can be formally reduced to the analysis of the \textit{persistent link complex} $\mathrm{Lk}_{\mathcal{K}}(v)$. 

Specifically, let $\Lk_i$ and $\Lk_j$ be the link complexes at indices $i$ and $j$ from the persistent link complex filtration. The chain isomorphism $\phi$ induces the following commutative diagram, extending the conjugacy in Theorem \ref{theorem:laplacian_conjugacy} to the persistent setting
\[
\xymatrix{
C_n(\mathcal{K}_i, \mathcal{K}_i^{\mathbf{v}}) \ar[rr]^-{\bar{f}^{i,j}_n} \ar[d]_{\phi}^{\cong} && C_n(\mathcal{K}_j, \mathcal{K}_j^{\mathbf{v}}) \ar[d]^{\phi}_{\cong} \\
C_{n-1}(\Lk_i) \ar[rr]^-{g^{i,j}_{n-1}} && C_{n-1}(\Lk_j),
}
\]
where $g^{i,j}_{n-1}$ is the transition morphism of the persistent link complex. Since $\phi$ is an isometry, it preserves the adjoint structure of the pullback differential, satisfying $(\delta^{i,j}_n)^* = \phi^{-1} \circ (\delta^{\Lk, i,j}_{n})^* \circ \phi$, where $\delta^{\Lk, i,j}_{n}$ is the generalized pullback differential acting on the link filtration.

Consequently, we obtain the following result for computing local persistence via links.

\begin{theorem}\label{theorem:persistence_link}
The $(i,j)$-persistent local Laplacian $\Delta_n^{i,j}$ is unitarily equivalent to the $(i,j)$-persistent Laplacian of the persistent link complex $\mathrm{Lk}_{\mathcal{K}}(v)$ at dimension $n-1$. Specifically,
\begin{equation*}
    \Delta_n^{i,j} = \phi^{-1} \circ \Delta_{n-1}^{\Lk,i,j} \circ \phi,
\end{equation*}
where $\Delta_{n-1}^{\Lk,i,j} \colon C_{n-1}(\Lk_i) \to C_{n-1}(\Lk_i)$ is defined as
\begin{equation*}
    \Delta_{n-1}^{\Lk,i,j} := \delta^{\Lk, i,j}_{n} (\delta^{\Lk, i,j}_{n})^* + (d^{\Lk, i}_{n-1})^* d^{\Lk, i}_{n-1}.
\end{equation*}
\end{theorem}

\begin{proof}
Recall that $\Delta_n^{i,j} = \delta^{i,j}_{n+1}(\delta^{i,j}_{n+1})^* + (d^i_n)^* d^i_n$. By Theorem~\ref{theorem:laplacian_conjugacy}, the relative boundary operator $d^i_n$ is conjugate to the link boundary operator $d^{\Lk, i}_{n-1}$ via $\phi$, i.e., $d^i_n = \phi^{-1} \circ d^{\Lk, i}_{n-1} \circ \phi$. 

Regarding the persistent components, the commutativity of the diagram implies $\bar{f}^{i,j}_n = \phi^{-1} \circ g^{i,j}_{n-1} \circ \phi$. Since $\phi$ is an isometry, its adjoint coincides with its inverse (i.e., $\phi^* = \phi^{-1}$), which yields the following identities
\begin{equation*}
    \delta^{i,j}_{n+1} = \phi^{-1} \circ \delta^{\Lk, i,j}_{n} \circ \phi \quad \text{and} \quad (\delta^{i,j}_{n+1})^* = \phi^{-1} \circ (\delta^{\Lk, i,j}_{n})^* \circ \phi.
\end{equation*}
Substituting these conjugacy relations into the definition of $\Delta_n^{i,j}$, we have
\begin{equation*}
    \Delta_n^{i,j} = \phi^{-1} \circ \Delta_{n-1}^{\Lk,i,j} \circ \phi.
\end{equation*}
The unitary equivalence follows directly from the fact that $\phi$ is an isometric isomorphism.
\end{proof}

Theorem \ref{theorem:persistence_link} establishes that the persistent local spectrum coincides with the persistent spectrum of the persistent link complex. This result reduces the computation of a relative persistent operator to a classical persistent Laplacian, thereby permitting the direct application of efficient algorithms for combinatorial persistent Laplacians.

\section{Persistent local Laplacian on datasets}\label{section:on_datasets}

The persistent local Laplacian can be naturally applied to discrete datasets, most notably point clouds and graph data. In this section, we detail the construction and computational methodology for persistent local Laplacians in these practical settings.

\subsection{Local Laplacian on point sets}

Let $X$ be a finite point set embedded in a metric space $(M, d)$. To analyze the local topological evolution around a specific point $x \in X$, we utilize the Vietoris-Rips filtration to construct a sequence of simplicial complexes.

\subsubsection{Vietoris-Rips filtration}

The \textit{Vietoris-Rips complex} at scale $r \geq 0$, denoted by $\mathcal{VR}(X, r)$, is defined as the simplicial complex whose $k$-simplices are formed by subsets of $k+1$ points in $X$ with a diameter at most $r$. By considering a strictly increasing sequence of parameters $0 = r_0 < r_1 < \dots < r_m$, we obtain a sequence of nested simplicial complexes
\begin{equation*}
    \mathcal{VR}(X, r_0) \hookrightarrow \mathcal{VR}(X, r_1) \hookrightarrow \dots \hookrightarrow \mathcal{VR}(X, r_m).
\end{equation*}
This sequence defines a persistent simplicial complex $\mathcal{VR}\colon (\mathbb{Z}, \leq) \to \mathbf{Simp}^{\hookrightarrow}$ in the category of simplicial complexes with inclusion morphisms.

\subsubsection{Persistent link complex on point sets}

For any fixed point $v \in X$, the local topological structure at $v$ is characterized by the relative chain complexes $C_{\ast}(\mathcal{K}_i, \mathcal{K}_i^{v})$, where $\mathcal{K}_i = \mathcal{VR}(X, r_i)$ denotes the Vietoris-Rips complex at scale $r_i$, and $\mathcal{K}_i^{v} = \{ \sigma \in \mathcal{K}_i \mid v \notin \sigma \}$ is the subcomplex consisting of all simplices not containing the vertex $v$.

To facilitate efficient computation, we invoke the chain isometry $\phi$ from Theorem~\ref{theorem:chain_isomorphism} and transition to the persistent link complex $\mathrm{Lk}_{\mathcal{VR}}(v) \colon (\mathbb{Z}, \leq) \to \mathbf{Simp}^{\hookrightarrow}$ of the Vietoris-Rips filtration. For each scale $r_i$, the link of $v$ is explicitly given by
\begin{equation*}
    \Lk_{\mathcal{K}, i}(v) = \{ \sigma \in \mathcal{VR}(X, r_i) \mid v \notin \sigma \text{ and } \sigma \cup \{v\} \in \mathcal{VR}(X, r_i) \}.
\end{equation*}
Let $N_i(v) = \{ u \in X \setminus \{v\} \mid d(u, v) \leq r_i \}$ be the neighborhood of $v$ at scale $r_i$. The correspondence between the local link and the neighborhood's geometry is formalized by the following proposition.

\begin{proposition}\label{prop:vr_link_equality}
For any $v\in X$, we have
\begin{equation*}
    \mathrm{Lk}_{\mathcal{VR}}(v) = \mathcal{VR}(N_{\bullet}(v), r_{\bullet})\colon (\mathbb{Z}, \leq) \to \mathbf{Simp}^{\hookrightarrow}.
\end{equation*}
Here, $\mathcal{VR}(N_{i}(v), r_{i})$ denotes the Vietoris-Rips complex of $N_{i}(v)$ at scale $r_i$.
\end{proposition}

\begin{proof}
It suffices to show that for any index $i$, the sets of simplices $\Lk_{\mathcal{K}, i}(v)$ and $\mathcal{VR}(N_i(v), r_i)$ are identical. By definition, $\sigma \in \Lk_{\mathcal{K}, i}(v)$ if and only if $v \notin \sigma$ and $\sigma \cup \{v\} \in \mathcal{VR}(X, r_i)$. In the Vietoris-Rips construction, the latter condition holds if and only if all vertices in $\sigma \cup \{v\}$ are pairwise within distance $r_i$. This requirement is equivalent to $\sigma \in \mathcal{VR}(N_i(v), r_i)$.

Since the transition morphisms in both filtrations are induced by the same inclusion maps on these identical simplex sets, the persistence simplicial complexes are equal.
\end{proof}

\subsubsection{Computational local spectrum via neighborhood filtrations}

Based on the identity established in Proposition~\ref{prop:vr_link_equality}, the calculation of the persistent local Laplacian can be reduced to the spectral analysis of a standard persistent Laplacian on the neighborhood filtration.

Let $X$ be a finite point set in a metric space $(M, d)$, and let $0 = r_0 < r_1 < \dots < r_m$ be an increasing sequence of filtration scales. For any fixed $v \in X$, we obtain a filtration of neighborhoods
\begin{equation*}
    N_{0}(v) \hookrightarrow N_{1}(v) \hookrightarrow \dots \hookrightarrow N_{m}(v),
\end{equation*}
where $N_k(v) = \{ u \in X \setminus \{v\} \mid d(u, v) \leq r_k \}$. This induces a filtration of Vietoris-Rips complexes
\begin{equation*}
    \mathcal{VR}(N_{0}(v), r_{0}) \hookrightarrow \mathcal{VR}(N_{1}(v), r_{1}) \hookrightarrow \dots \hookrightarrow \mathcal{VR}(N_{m}(v), r_{m}).
\end{equation*}

Let $C^i_{\ast} = C_{\ast}(\mathcal{VR}(N_i(v), r_i))$ denote the associated chain complex at scale $r_i$. For any pair of indices $i \leq j$, the inclusion of simplicial complexes induces an injective chain map $f^{i,j}_{\ast} \colon C^i_{\ast} \hookrightarrow C^j_{\ast}$. In this context, the subspace $\Theta^{i,j}_{n}$ defined in Section~\ref{section:generalized_Laplacian} can be reduced to the form
\begin{equation*}
    \Theta^{i,j}_{n} = \{ x \in C_{n}^j \mid d^j_{n} x \in \mathrm{im}(f^{i,j}_{n-1}) \cong C_{n-1}^i \}.
\end{equation*}
The persistent structure is then characterized by the following sequence of operators
\begin{equation*}
    \xymatrix{
        \Theta^{i,j}_{n} \ar@{->}[r]^{\delta^{i,j}_{n}} & C_{n-1}^{i} \ar@{->}[r]^{d^i_{n-1}} & C_{n-2}^{i},
    }
\end{equation*}
where the persistent boundary operator $\delta^{i,j}_{n}$ is defined as the composition of the following maps
\begin{equation*}
    \xymatrix{
        \Theta^{i,j}_{n} \ar@{^{(}->}[r]^{\iota} & C_{n}^{j} \ar@{->}[r]^{d^j_{n}} & C_{n-1}^{j} \ar@{->}[r]^{(f^{i,j}_{n-1})^\ast} & C_{n-1}^{i}.
    }
\end{equation*}
The $(i,j)$-persistent Laplacian of the link complex, $\Delta_{n-1}^{i,j} \colon C_{n-1}^{i} \to C_{n-1}^{i}$, is then given by
\begin{equation*}
    \Delta_{n-1}^{i,j} = \delta^{i,j}_{n}(\delta^{i,j}_{n})^{\ast} + (d^i_{n-1})^{\ast}d^i_{n-1}.
\end{equation*}

Crucially, the $(i,j)$-persistent local Laplacian of $X$ at $v$, $\Delta_n^{v,i,j}$, is unitarily equivalent to the persistent Laplacian $\Delta_{n-1}^{i,j}$ computed over the neighborhood filtration. Consequently, we have the following spectral identity
\begin{equation*}
    \mathbf{Spec}(\Delta_n^{v,i,j}) = \mathbf{Spec}(\Delta_{n-1}^{i,j}), \quad \text{for } n \geq 1.
\end{equation*}
This spectral equivalence ensures that the multi-scale local topological signatures of $v$ can be efficiently extracted by analyzing the neighbor-induced sub-complexes.
pl
The localization allows the persistent local Laplacian to act as a high-resolution, multi-scale descriptor for large-scale datasets, where computing global persistence would be computationally prohibitive.

\subsubsection{The 0-th persistent local Laplacian}

Proposition \ref{proposition:zero_Laplacian} shows that the local Laplacian coincides with the degree of the vertex. We now proceed to study the $0$-th local Laplacian in the persistent setting.

\begin{proposition}
The $0$-th persistent local Laplacian $\Delta_0^{v,i,j} \colon C_0(\mathcal{K}_i, \mathcal{K}_i^v) \to C_0(\mathcal{K}_i, \mathcal{K}_i^v)$ is given by
\begin{equation*}
    \Delta_0^{v,i,j} = \deg_{r_j}(v),
\end{equation*}
where $\deg_{r_j}(v)$ denotes the number of neighbors of $v$ within distance $r_j$.
\end{proposition}

\begin{proof}
By the definition of the persistent local Laplacian, for the case $n=0$, the lower-order term $(d_{0}^i)^* d_{0}^i$ vanishes. Consequently, we have
\begin{equation*}
    \Delta_0^{v,i,j} = \bar{\delta}_1^{i,j} (\bar{\delta}_1^{i,j})^*.
\end{equation*}
Here, $\bar{\delta}_1^{i,j}$ is the pullback differential between relative chain complexes.
Recall that the domain of the persistent boundary operator is defined as
\begin{equation*}
    \Theta^{i,j}_{1} = \{ x \in C_1(\mathcal{K}_j, \mathcal{K}_j^v) \mid d^j_{1} x \in C_0(\mathcal{K}_i, \mathcal{K}_i^v) \}.
\end{equation*}
Consider a basis element of the relative chain group $C_1(\mathcal{K}_j, \mathcal{K}_j^v)$, which is a relative $1$-simplex (edge) $[v, w]$ such that $d(v, w) \leq r_j$. Its boundary in the chain complex is $d_1^j([v, w]) = [w] - [v]$. In the relative space $C_0(\mathcal{K}_i, \mathcal{K}_i^v)$, any vertex $[w]$ for $w \neq v$ represents the zero class. Thus, we have
\begin{equation*}
    d_1^j([v, w]) = -[v].
\end{equation*}
This implies that $\Theta^{i,j}_{1}$ is linearly spanned by all relative $1$-simplices in $\mathcal{K}_j$ that are incident to $v$.

The dimension of $\Theta^{i,j}_{1}$ is therefore precisely $\deg_{r_j}(v)$. With respect to the standard basis for $\Theta^{i,j}_{1}$ and the basis $\{[v]\}$ for $C_0^i$, the persistent boundary operator $\bar{\delta}_1^{i,j}$ is represented by a $1 \times \deg_{r_j}(v)$ row vector
\begin{equation*}
    M_{\bar{\delta}_1^{i,j}} = \begin{pmatrix} -1 & -1 & \dots & -1 \end{pmatrix}.
\end{equation*}
Then the matrix representation of the persistent Laplacian is
\begin{equation*}
    M_{\Delta_0^{v,i,j}} = M_{\bar{\delta}_1^{i,j}} (M_{\bar{\delta}_1^{i,j}})^T =  \deg_{r_j}(v).
\end{equation*}
This identity $\Delta_0^{v,i,j} = \deg_{r_j}(v)$ completes the proof.
\end{proof}

\subsection{Local spectral analysis on graphs}

Spectral analysis is a cornerstone for processing weighted networks, yet global decomposition becomes computationally prohibitive for large-scale graphs. This limitation necessitates a shift toward localized spectral methods. Our persistent local Laplacian addresses this challenge by restricting spectral analysis to the filtration of neighbor-induced subgraphs. This tool effectively transforms memory-intensive global operations into a scalable, multi-scale localized computational task, permitting the efficient extraction of high-resolution topological signatures from massive networks.

\subsubsection{Persistent local invariants on weighted graphs}

Let $G=(V, E, w)$ be a weighted graph with vertex set $V$, edge set $E \subseteq V \times V$, and a weight function $w\colon  E \to \mathbb{R}$. For any real number $r$, we consider the subgraph
\begin{equation*}
    G_r = (V, E_r), \quad E_r = \{e \in E \mid w(e) \leq r\}.
\end{equation*}
This defines a filtration of clique complexes (also known as flag complexes). For any sequence of scales $0 = r_0 < r_1 < \dots < r_m$, we obtain a nested sequence of simplicial complexes
\begin{equation*}
    \Clq(G_{r_0}) \hookrightarrow \Clq(G_{r_1}) \hookrightarrow \dots \hookrightarrow \Clq(G_{r_m}).
\end{equation*}
Similar to the constructions in the previous section, for any vertex $v \in V$ and indices $0 \leq i \leq j \leq m$, we can obtain the  persistent local homology $H^{i,j}_{n}(G,v)$ and the persistent local Laplacian $\Delta^{i,j}_{n}(G,v)$ based on the relative chain complexes.

Alternatively, the local structure can be analyzed via the link complex of the vertex. The following proposition provides the graph-theoretic counterpart to the link identity:

\begin{proposition}\label{prop:graph_link_clique}
For any vertex $v \in V$, let $G_i[N_i(v)]$ be the subgraph induced by the set of neighbors $N_i(v) = \{ u \in V \mid (u, v) \in E_i \}$. The persistent link complex of $v$ is identically equal to the persistent clique complex
\begin{equation*}
    \mathrm{Lk}_{\Clq(G)}(v) = \Clq(G_{\bullet}[N_{\bullet}(v)]) \colon (\mathbb{Z}, \leq) \to \mathbf{Simp}^{\hookrightarrow} .
\end{equation*}
\end{proposition}

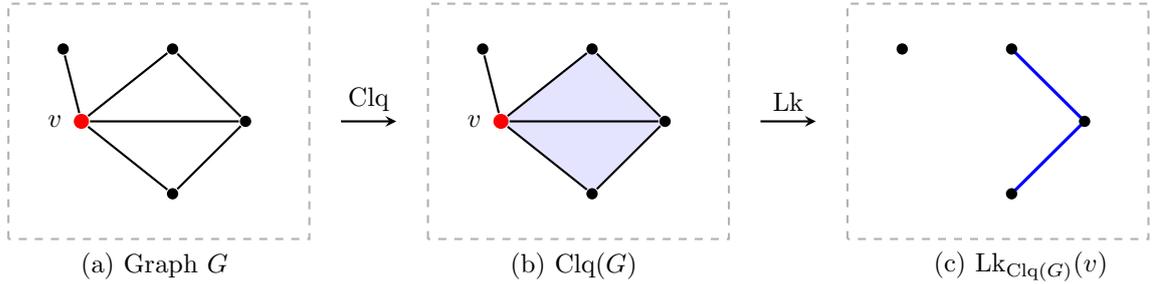
\begin{figure}[htbp]
  \centering
  \begin{tikzpicture}[
    scale=1.2, thick,
    node/.style={circle, fill=black, inner sep=1.5pt},
    v_node/.style={circle, fill=red, inner sep=2pt, label=left:{\small $v$}},
    neighbor/.style={circle, fill=black, inner sep=1.5pt}, 
    edge/.style={thick, black},
    link_edge/.style={very thick, blue},
    face_clique/.style={fill=blue!15, opacity=0.7},
    arrow/.style={->, >=stealth, thick, shorten >=5pt, shorten <=5pt}
  ]

    \begin{scope}
        \draw[dashed, gray!60] (-0.8, -1.3) rectangle (2.5, 1.3);
        
        \node[v_node] (v) at (0,0) {};
        \node[neighbor] (n1) at (1,0.8) {};
        \node[neighbor] (n2) at (1,-0.8) {};
        \node[neighbor] (n3) at (1.8,0) {};
        \node[node] (o) at (-0.2,0.8) {}; 
        
        \draw[edge] (v) -- (n1);
        \draw[edge] (v) -- (n2);
        \draw[edge] (v) -- (n3);
        \draw[edge] (n1) -- (n3);
        \draw[edge] (n2) -- (n3);
        \draw[edge] (v) -- (o);
        
        \node at (0.8, -1.6) {\small (a) Graph $G$};
    \end{scope}

    \draw[arrow] (2.7, 0) -- (3.6, 0) node[midway, above] {\small $\mathrm{Clq}$};

    \begin{scope}[xshift=4.6cm]
        \draw[dashed, gray!60] (-0.8, -1.3) rectangle (2.5, 1.3);

        \fill[face_clique] (0,0) -- (1,0.8) -- (1.8,0) -- cycle;
        \fill[face_clique] (0,0) -- (1,-0.8) -- (1.8,0) -- cycle;
        
        \node[v_node] (v2) at (0,0) {};
        \node[neighbor] (n1b) at (1,0.8) {};
        \node[neighbor] (n2b) at (1,-0.8) {};
        \node[neighbor] (n3b) at (1.8,0) {};
        \node[node] (ob) at (-0.2,0.8) {};
        
        \draw[edge] (v2) -- (n1b);
        \draw[edge] (v2) -- (n2b);
        \draw[edge] (v2) -- (n3b);
        \draw[edge] (n1b) -- (n3b);
        \draw[edge] (n2b) -- (n3b);
        \draw[edge] (v2) -- (ob);

        \node at (0.8, -1.6) {\small (b) $\mathrm{Clq}(G)$};
    \end{scope}

    \draw[arrow] (7.3, 0) -- (8.2, 0) node[midway, above] {\small $\mathrm{Lk}$};

    \begin{scope}[xshift=9.2cm]
        \draw[dashed, gray!60] (-0.8, -1.3) rectangle (2.5, 1.3);

        \draw[link_edge] (1,0.8) -- (1.8,0);
        \draw[link_edge] (1,-0.8) -- (1.8,0);
        
        \node[neighbor] (n1c) at (1,0.8) {};
        \node[neighbor] (n2c) at (1,-0.8) {};
        \node[neighbor] (n3c) at (1.8,0) {};
        
        \node[node] (oc) at (-0.2,0.8) {}; 
      
        \node at (1.1, -1.6) {\small (c) $\mathrm{Lk}_{\mathrm{Clq}(G)}(v)$};
    \end{scope}
  \end{tikzpicture}
  \caption{Relationship between a graph and its local persistent structure. (b) The clique complex $\Clq(G)$ highlights connected substructures in blue. (c) The link complex focuses on the relational structure between neighbors by removing vertex $v$.}
  \label{fig:clique_link_final}
\end{figure}

For any $i \leq j$, let $H^{i,j}_n(\Clq(G), v)$ and $\Delta^{i,j}_n(\Clq(G), v)$ denote the persistent homology and persistent Laplacian of the persistent link complex $\Clq(G_{\bullet}[N_{\bullet}(v)])$. Based on the correspondence between the relative complex and the link complex, we obtain the following structural equivalence.

\begin{proposition}
Let $G=(V, E, w)$ be a weighted graph. For any vertex $v \in V$ and $i \leq j$, we have
\begin{enumerate}[label=$(\roman*)$]
    \item $H^{i,j}_{n}(G,v) \cong H^{i,j}_{n-1}(\Clq(G), v)$ for $n\geq 1$;
    \item $\Delta^{i,j}_{n}(G,v)$ is unitarily equivalent to $\Delta^{i,j}_{n-1}(\Clq(G), v)$ for $n \geq 1$.
\end{enumerate}
\end{proposition}
\begin{proof}
The result is a direct consequence of Theorem \ref{theorem:persistence_link} and Proposition \ref{prop:graph_link_clique}.
\end{proof}

\subsubsection{Computation of low-dimensional persistent Laplacians on graphs}

While the $0$-dimensional (persistent) local Laplacian on a weighted graph is given by the vertex degree, we shift our focus to the $1$-dimensional local Laplacian to analyze the local connectivity and cycle structures around a vertex.

By virtue of the unitary equivalence between the relative Laplacian and the link Laplacian, the computation of the $1$-dimensional local Laplacian $\Delta_1(G, v)$ is effectively reduced to the $0$-dimensional combinatorial Laplacian of the link complex $\Lk_K(v)$. In the context of a clique complex, this link is identically the clique complex of the induced neighborhood subgraph $G[N(v)]$, where $N(v) = \{x \in V \mid (x, v) \in E\}$.

\begin{proposition}\label{prop:matrix_rep}
Let $G=(V, E)$ be a graph. For a fixed vertex $v \in V$, the $1$-dimensional local Laplacian $\Delta_1(G, v)$ is unitarily equivalent to the $0$-dimensional graph Laplacian $L$ of the induced subgraph $G[N(v)]$. With respect to the standard basis indexed by the neighbors $\{u_1, u_2, \dots, u_d\} \subseteq N(v)$, the representation matrix $L = (L_{pq})$ is given by
\begin{equation*}
    L_{pq} = \begin{cases} 
        \deg_{G[N(v)]}(u_p) & \text{if } p=q, \\
        -1 & \text{if } p \neq q \text{ and } (u_p, u_q) \in E, \\
        0 & \text{otherwise},
    \end{cases}
\end{equation*}
where $\deg_{G[N(v)]}(u_p)$ denotes the degree of vertex $u_p$ within the induced subgraph $G[N(v)]$.
\end{proposition}
\begin{proof}
The result follows from the fact that the $0$-th combinatorial Laplacian of the clique complex $\Clq(G[N(v)])$ coincides exactly with the standard graph Laplacian of $G[N(v)]$.
\end{proof}

Consider a filtration of graphs where $G_i = (V_i, E_i)$ is a subgraph of $G_j = (V_j, E_j)$. Let $C^i_{\ast}$ and $C^j_{\ast}$ denote the chain complexes of the clique complexes associated with $G_i$ and $G_j$, respectively. Then we have
\begin{equation*}
  \Theta_{i,j} = \{x \in C^j_{1} \mid d_1^j x \in C^i_{0}\},
\end{equation*}
where $d_1^j\colon  C^j_1 \to C^j_0$ is the boundary operator of the complex $C^j_{\ast}$. We define the pullback differential $\delta^{i,j}_1\colon  \Theta_{i,j} \to C^i_{0}$ as the restriction of the boundary operator, namely $\delta^{i,j}_1 x = d_1^j x$. The $1$-dimensional persistent local Laplacian is then defined as
\begin{equation*}
  \Delta^{i,j}_1 = \delta^{i,j}_1(\delta^{i,j}_1)^{\ast}\colon  C^i_{0} \to C^i_{0}.
\end{equation*}

To perform the computation, let $B_j \in \mathbb{R}^{|V_j| \times |E_j|}$ be the incidence matrix of $G_j$. We partition the rows of $B_j$ according to the vertex sets $V_i$ and $V_{out} = V_j \setminus V_i$. Then, we write
\begin{equation*}
    B_j = \begin{pmatrix} B_{in} \\ B_{out} \end{pmatrix},
\end{equation*}
where $B_{in}$ and $B_{out}$ are the submatrices consisting of rows corresponding to vertices in $V_i$ and $V_{out}$, respectively.

\begin{proposition}\label{prop:persistent_laplacian_matrix}
Let $G_i \subseteq G_j$ be as defined above. The persistent local Laplacian $\Delta^{i,j}_1 = \delta^{i,j}_1(\delta^{i,j}_1)^{\ast}$ in the standard basis of $C^i_0$ has the representation matrix
\begin{equation*}
    L^{i,j} = B_{in} P_{\Theta} B_{in}^T,
\end{equation*}
where $P_{\Theta}$ is the orthogonal projection matrix of $C^{j}_{1}$ onto the subspace $\Theta_{i,j}$. Specifically, $P_{\Theta}$ can be computed via the Moore-Penrose pseudoinverse as
\begin{equation*}
    P_{\Theta} = I - B_{out}^T (B_{out} B_{out}^T)^\dagger B_{out}.
\end{equation*}
\end{proposition}

\begin{proof}

The pullback differential $\delta_1^{i,j}$ is the restriction of the boundary operator $d_1^j$ to $\Theta_{i,j}$. For any $x \in \Theta_{i,j}$, its image lies in $C_0^i$, so we can write the action as $\delta_1^{i,j} x = B_{in} x$.
The adjoint $(\delta_1^{i,j})^*\colon  C_0^i \to \Theta_{i,j}$ is given by $P_{\Theta} B_{in}^T$, where $P_{\Theta}$ is the orthogonal projection of $C^{j}_{1}$ onto $\Theta_{i,j}$.

Thus, the persistent local Laplacian is
\begin{equation*}
    \Delta_1^{i,j} = \delta_1^{i,j} (\delta_1^{i,j})^* = B_{in} P_{\Theta} B_{in}^T.
\end{equation*}
Using the property of orthogonal projections onto the kernel of a matrix, we have $P_{\Theta} = I - B_{out}^T (B_{out} B_{out}^T)^\dagger B_{out}$, where $\dagger$ denotes the Moore-Penrose pseudoinverse.
Therefore, the representation matrix is
\begin{equation*}
    L^{i,j} = B_{in} \left( I - B_{out}^T (B_{out} B_{out}^T)^\dagger B_{out} \right) B_{in}^T,
\end{equation*}
as desired.
\end{proof}

\begin{example}
Let $G_j = (V_j, E_j)$ be a graph with $V_j = \{u_1, u_2, v_1, v_2\}$ and $V_i = \{u_1, u_2\}$. The edge set $E_j$ consists of a direct internal edge $e_1 = (u_2, u_1)$ and an external bridge path of length $3$ formed by $e_2 = (v_1, u_1)$, $e_3 = (v_2, v_1)$, and $e_4 = (u_2, v_2)$. 

\begin{figure}[htbp]
    \centering
    \begin{tikzpicture}[
        scale=1.8,
        thick, 
        main node/.style={circle, draw, thick, fill=red!15, font=\large, minimum size=0.7cm},
        ext node/.style={circle, draw, thick, fill=blue!15, font=\large, minimum size=0.7cm},
        >={Stealth[length=2mm]}
    ]

        \node[main node] (u1) at (0,0) {$u_1$};
        \node[main node] (u2) at (1.8,0) {$u_2$};
        \node[ext node] (v1) at (0,1.2) {$v_1$};
        \node[ext node] (v2) at (1.8,1.2) {$v_2$};

        \path[draw, thick, ->]
            (u2) edge node[below, font=\large] {$e_1$} (u1)
            (v1) edge node[left, font=\large] {$e_2$} (u1)
            (v2) edge node[above, font=\large] {$e_3$} (v1)
            (u2) edge node[right, font=\large] {$e_4$} (v2);

        \draw[rounded corners, blue, thick, dashed] 
            (-0.4,-0.4) rectangle (2.2,0.4);
        \node[font=\large] at (2.6,0) {$G_i$};
        \node[font=\large] at (2.6,1.2) {$G_j$};

    \end{tikzpicture}
    \caption{Illustration of the graphs $G_i$ and $G_j$.}
\end{figure}
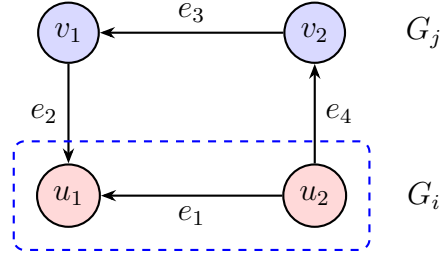

The boundary operator $d_1^j$ is represented by the incidence matrix $B$ and we have
\begin{equation*}
 d_1^j \begin{pmatrix}
         e_1 & e_2 & e_3 & e_4 
       \end{pmatrix} = \begin{pmatrix}
         u_1 & u_2 & v_1 & v_2 
       \end{pmatrix}B.
\end{equation*}
Here,
\begin{equation*}
B = \begin{pmatrix} B_{in} \\ B_{out} \end{pmatrix} = \left( \begin{array}{cccc} 1 & 1 & 0 & 0 \\ -1 & 0 & 0 & -1 \\ \hline 0 & -1 & 1 & 0 \\ 0 & 0 & -1 & 1 \end{array} \right).
\end{equation*}
The subspace $\Theta_{i,j}$ is defined as the kernel of $B_{out}$. Solving $B_{out} \mathbf{x} = 0$ yields a $2$-dimensional subspace spanned by the orthogonal basis:
\begin{itemize}
    \item $\xi_1 = (1, 0, 0, 0)^T$ (representing the internal edge $e_0$);
    \item $\xi_2 = (0, 1, 1, 1)^T$ (representing the chain of the external path $e_1 + e_2 + e_3$).
\end{itemize}
The orthogonal projection matrix onto $\Theta_{i,j}$ is given by $P_{\Theta} = \frac{\xi_1 \xi_1^T}{\|\xi_1\|^2} + \frac{\xi_2 \xi_2^T}{\|\xi_2\|^2}$. The persistent local Laplacian $L^{i,j}$ is computed via the product $B_{in} P_{\Theta} B_{in}^T$. The images of the basis vectors under $B_{in}$ are
\begin{equation*}
    B_{in} \xi_1 = \begin{pmatrix} 1 \\ -1 \end{pmatrix}, \quad B_{in} \xi_2 = \begin{pmatrix} 1 \\ -1 \end{pmatrix}.
\end{equation*}
Hence, we obtain
\begin{equation*}
    L^{i,j} = \frac{(B_{in} \xi_1)(B_{in} \xi_1)^T}{1} + \frac{(B_{in} \xi_2)(B_{in} \xi_2)^T}{3} = \frac{4}{3} \begin{pmatrix} 1 & -1 \\ -1 & 1 \end{pmatrix}.
\end{equation*}
The increase in the spectral gap of $L^{i,j}$ from $2$ (intrinsic to $G_i$) to $8/3$ indicates that the persistent Laplacian effectively captures the enhanced connectivity between $u_1$ and $u_2$ afforded by the external path in $G_j$.
\end{example}

\subsubsection{Spectral interpretation of persistent local Laplacian}

In the $0$-dimensional case, the combinatorial Laplacian coincides with the classical graph Laplacian $L$. In spectral graph theory, $L$ characterizes the diffusion process over a network. The multiplicity of the zero eigenvalue corresponds to the $0$-th Betti number, indicating the number of connected components. Notably, the second smallest eigenvalue, known as the \textit{Fiedler value} or \textit{algebraic connectivity}, quantifies the robustness of the graph's connectivity. This spectral gap is intimately related to the \textit{Cheeger invariant}, providing a discrete analogue to the geometric bottleneck ratio.

The combinatorial Laplacian generalizes these notions to higher-dimensional simplicial complexes. According to the Hodge Theorem for complexes, the harmonic space of the $k$-dimensional Laplacian $\Delta_k$ is isomorphic to the $k$-th simplicial homology group $H_k(K; \mathbb{R})$. Beyond the kernel, the non-zero spectrum reflects the intrinsic geometry of the complex, such as the "tightness" or the degree of enclosure of $k$-dimensional cavities.

In the persistent setting, the persistent Laplacian admits a persistent version of the Hodge decomposition. Its harmonic part is naturally isomorphic to the persistent homology groups, a central object in topological data analysis. The non-harmonic spectrum, meanwhile, characterizes the evolution of diffusion rates and the ``flow'' of cavity closure across different scales of the filtration.

The local Laplacian proposed in this work characterizes the local geometric features in the neighborhood of a specific vertex. Specifically, its harmonic space is isomorphic to the corresponding local homology group. The spectrum of the $0$-dimensional local Laplacian reflects the vertex degree, while the $1$-dimensional local spectrum characterizes the connectivity and diffusion efficiency among its neighbors. Higher-dimensional local spectra describe the presence and relative closure of local cavities surrounding the vertex.

Finally, the persistent local Laplacian extends this to a multi-scale setting. There exists an isomorphism between the harmonic space of the persistent local Laplacian and the corresponding persistent local homology. Its spectral evolution provides a refined description of the local space, capturing the dynamic changes in diffusion velocity and the topological flux of local enclosure across the filtration.

In essence, while the classical Laplacian is a fundamental operator for characterizing topological and geometric information, the persistent local Laplacian represents its evolution towards higher-dimensional analysis, dynamic multi-scale tracking, and localized geometric description.

\section{Acknowledgments}

This work was supported in part by the Natural Science Foundation of China [Grant Number 12401080], the Scientific Research Foundation of Chongqing University of Technology, the Start-up Research Fund of the University of North Carolina at Charlotte, the National Key Research and Development Program of China (NKPs) [Grant Number 2024YFA1013201], the Natural Science Foundation of Chongqing (NSFCQ) [Grant Number CSTB2024NSCQ-LZX0040], the Special Project of Chongqing Municipal Science and Technology Bureau [Grant Number 2025CCZ015].

\bibliographystyle{plain}  
\bibliography{Reference}

\end{document}